\documentclass[10pt]{amsart}
\usepackage{amssymb,amsmath,amsopn,textcomp, thmtools}

\usepackage[utf8]{inputenc}
\usepackage[T1]{fontenc}
\usepackage[english]{babel}
\usepackage[centertags]{amsmath}
\usepackage{nicefrac, esint}
\usepackage{mathrsfs}

\usepackage{tikz}
\usetikzlibrary{shapes,trees}

\usepackage{caption, float}

\usepackage{enumitem}

\usepackage[colorlinks=true, linkcolor=blue, citecolor=black]{hyperref}

\addto\extrasenglish{}
\addto\extrasenglish{}

\newtheorem{theorem}{Theorem}[section]
\newtheorem*{theorem*}{Theorem}
\newtheorem{corollary}[theorem]{Corollary}
\newtheorem{lemma}[theorem]{Lemma}
\newtheorem{proposition}[theorem]{Proposition}

\theoremstyle{definition}
\newtheorem{definition}[theorem]{Definition}
\newtheorem{example}[theorem]{Example}
\newtheorem{remark}[theorem]{Remark}

\newcommand{\iR}{\mathbb{R}}
\newcommand{\iN}{\mathbb{N}}
\newcommand{\iZ}{\mathbb{Z}}
\newcommand{\Z}{\mathbb{Z}}

\usepackage{bbm}
\newcommand{\R}{\mathbb{R}}
\newcommand{\N}{\mathbb{N}}

\renewcommand{\d}{\operatorname{d}\!}

\usepackage[backgroundcolor=white, bordercolor=blue,
linecolor=blue]{todonotes}
\numberwithin{equation}{section}

\title{Discrete versions of the Li-Yau gradient estimate}

\author{Dominik Dier}
\email{dominik.dier@uni-ulm.de}
\address{Institut für Angewandte Analysis, 
Universit\"{a}t Ulm, Helmholtzstraße 18, D-89081 Ulm, Germany}

\author{Moritz Kassmann}
\email{moritz.kassmann@uni-bielefeld.de}
\address{Fakult\"{a}t f\"{u}r Mathematik, Postfach 100131, D-33501 Bielefeld, 
Germany}

\author{Rico Zacher}
\email{rico.zacher@uni-ulm.de}
\address{Institut für Angewandte Analysis, Universit\"{a}t Ulm, Helmholtzstraße 
18, D-89081 Ulm, Germany}



\begin{document}

\begin{abstract}
We study positive solutions to the heat equation on graphs. We prove variants 
of the Li-Yau gradient estimate and the differential Harnack inequality. For 
some graphs, we can show the estimates to be sharp. We 
establish new computation rules for differential operators on discrete spaces 
and introduce a relaxation function that governs the time dependency in the 
differential Harnack estimate. 
\end{abstract}

\maketitle

\section{Introduction}

The heat equation plays a fundamental role in several fields of Mathematics and 
provides a link between Analysis, Stochastics and Geometry. It has been 
intensively studied on different state spaces, e.g., the 
Euclidean space, Riemannian manifolds, and general metric measure spaces. 
In this work, we study pointwise estimates for positive solutions to the heat 
equation on graphs. We aim at precise results whenever this 
is possible. If the graph under consideration is small, i.e., if it contains 
only few vertices, then we check our estimates by explicit computation. For 
the sequence of graphs given by $(\tau \Z^d)_{\tau>0}$ we try to trace 
the influence 
of 
the parameter $\tau \to 0+$. This allows us to compare the 
estimates  with well-known results for the limit space $\R^d$.

\medskip

Before we explain the framework of our study in greater detail, let us review 
some fundamental results with regard to the heat equation on Riemannian 
manifolds. The classical gradient estimate given by Li-Yau \cite{LiYa86} 
holds true for positive solutions $u: [0, \infty) \times M \to (0, \infty)$ of 
the heat equation $\partial_t u - \Delta u = 0$ on a complete $d$-dimensional 
Riemannian manifold $M$ with $Ric(M) \geq 0$:
\begin{align}\label{eq:diffHarnack-classicalu}
\frac{| \nabla u (t,x)|^2}{u^2(t,x)} - \frac{\partial_tu(t,x)}{u(t,x)} \leq 
\frac{d}{2t} \qquad (t>0, x \in M)\,,
\end{align}
or, equivalently,
\begin{align}\label{eq:diffHarnack-classical}
|\nabla \log u (t,x)|^2 - \partial_t (\log u)(t,x) \leq \frac{d}{2t} \qquad ( 
t >0, x \in M)\,.
\end{align}

An important consequence of this estimate is a pointwise bound on the solution 
itself, which
can be obtained from integration over a path that connects two given points 
$(t_1,x_1)$ and
$(t_2, x_2)$ with $t_2 > t_1>0$:
\begin{align}\label{eq:Harnack-classical}
u(t_1,x_1) \leq u(t_2,x_2)\left( \frac{t_2}{t_1} \right)^{d/2} \exp \left( 
\frac{\rho^2(x_1,x_2)}{4(t_2-t_1)} \right).
\end{align}

Note that estimates \eqref{eq:diffHarnack-classicalu}, \eqref{eq:diffHarnack-classical}, and \eqref{eq:Harnack-classical} are sharp in the sense that corresponding 
equalities hold
true for the fundamental solution to the heat equation on $\mathbb{R}^d$, i.e., 
if $u(t,x)$
equals $(4 \pi t)^{-d/2} \exp \left(\frac{-|x|^2}{4t}\right)$.

\medskip

The aim of the current project is to study estimates of the type 
\eqref{eq:diffHarnack-classicalu}, \eqref{eq:diffHarnack-classical}, and \eqref{eq:Harnack-classical}
for positive
solutions to the heat equation on graphs. In order to establish a corresponding 
theory, we
establish new computation rules for functions defined on discrete spaces. 

\medskip

Let $G=(V,E)$ be a graph. All graphs appearing in this work are assumed to 
be undirected. For two vertices $x,y\in V$ we 
write $x\sim y$ if 
there is an edge between $x$ and $y$,
that is, $xy\in E$. We allow for edge weights; the weight of the edge $xy$ 
from $x$ to $y$ is denoted by $w_{xy}$  
and is always assumed to be positive. Moreover, we assume that the graph is  
locally finite, i.e., for every $x\in V$ the  
set of all $y\in V$ with $y\sim x$ is finite.

\medskip 

Set $\iR^V:=\{u:\,V\to \iR\}$ and assume $\mu:\,V\to (0,\infty)$. We consider 
the generalized Laplacian on $G$, which is the operator $\Delta: \iR^V\to 
\iR^V$ defined by 

\begin{align} \label{genlaplace}
\Delta u(x)=\,\frac{1}{\mu(x)}\sum_{y\sim x}w_{xy}\big(u(y)-u(x)\big) \qquad 
(x\in V)\,.
\end{align}

We will also use the operator $L:=-\Delta$. We say that a function 
$u:[0,\infty) \times V \to \R$ solves the heat equation on $G$ if $\partial_t 
u - \Delta u = 0$ on $[0,\infty) \times V$. 
We recall the definition of $\Gamma, \Gamma_2 : \mathbb{R}^V 
\times \mathbb{R}^V \to \mathbb{R}^V$:

\begin{align*}
2 \Gamma(v,w) &= \Delta(v w) - v \Delta w - w 
\Delta v \,, \\
2 \Gamma_2(v,w) &= \Delta(\Gamma(v,w)) - \Gamma(v,\Delta w) - \Gamma(w, \Delta 
v) \,.
\end{align*}

As it is usual, we write $\Gamma(v)$ instead of $\Gamma(v,v)$ and analogously 
$\Gamma_2(v)$ instead of $\Gamma_2(v,v)$. A crucial identity in the 
classical approach to Li-Yau estimates is

\begin{align}\label{eq:chain-log}
 \Delta (\log u) = \frac{\Delta u}{u} - |\nabla \log u|^2
\end{align}

for positive functions $u:M \to \R$. The equality follows directly from the 
chain rule. One way to compensate the lack of the chain rule for differences is 
provided in \cite{BHLLMY15}. Instead of \eqref{eq:chain-log}, the authors invoke 
the identity

\begin{align}
 2 \sqrt{u} \Delta (\sqrt{u}) = \Delta u - 2 \Gamma(\sqrt{u}) \,,
\end{align}

which holds true on graphs, too. This equality allows to derive estimates of 
$\Gamma(\sqrt{u})$ if $u$ is a 
positive solution to the heat equation. In the present work, we suggest to 
follow another path. We provide a discrete version of \eqref{eq:chain-log} and 
show for positive functions 
$u:V \to \R$

\begin{align} \label{eq:log-equality}
\Delta (\log u) = \frac{\Delta u}{u} - \Psi_\Upsilon (\log 
u) \,,
\end{align}

where $\Psi_\Upsilon (v)(x)=\,\frac{1}{\mu(x)}\sum_{y\sim x}w_{xy} 
\Upsilon\big(v(y)-v(x)\big)$ and $\Upsilon(z) = 
e^z - z - 1$. Note $\Psi_{\Upsilon} (\log u)$ equals $\Gamma^{\log} u$ as in 
\cite[Section 3]{Mue14}. In \autoref{sec:fund-id} we provide more general 
computation rules. Note that $v \mapsto \Psi_\Upsilon (v)$ is a replacement of 
the quadratic function $v \mapsto \Gamma(v)$ and  $z \mapsto \Upsilon(z) = e^z 
- z - 1$ replaces the square function in the 
expression $|\nabla \log u|^2$.

\medskip

One of our main results is a Li-Yau type inequality for positive solutions $u$ 
to the heat equation on a 
finite connected undirected graph $G=(V,E)$:

\begin{theorem} \label{theo:finite-graphs}
Assume that $G$ satisfies CD($F$;0) and let 
$\varphi$ be the relaxation function 
associated with the CD-function $F$. 
Suppose that $u:[0,\infty)\times V\rightarrow (0,\infty)$ is a solution of the 
heat equation
on $G$. Then
\begin{align} \label{eq:LYfinite}
-\Delta (\log u) (t,x) \le \varphi(t)\quad \mbox{in}\;(0,\infty)\times V,
\end{align}
and thus
\begin{align} \label{eq:diffHarnack}
\Psi_\Upsilon(\log u)(t,x) - \partial_t (\log u)(t,x) \leq \varphi(t)\quad 
\mbox{in}\;(0,\infty)\times V.
\end{align}
\end{theorem}

The condition CD($F$;0) is formulated 
locally at each point $x \in V$ and involves only neighbors of second order, 
cf. \autoref{def:CD}. We use the abbreviation CD as in ``curvature 
dimension'' although the relation to classical CD-conditions like $\Gamma_2(f) 
\geq \frac{1}{d}(\Delta f)^2$ for all $f$ or more recent related conditions from 
\cite{BHLLMY15}, \cite{Mue14} (like e.g.\ the so-called exponential curvature dimension inequality CDE($n$,0)) has not yet been established. The examples from 
\autoref{sec:curv-ineq} suggest that there is a close relation between 
CD($F$;0) and other conditions from the literature. Note that $F:[0, \infty) 
\to [0, \infty)$ is a continuous function such that $F(0)=0$, $F(x)/x$ is 
strictly 
increasing, and $1/F$ is integrable at $+\infty$. The relaxation function 
$\varphi$ is the unique positive solution to $\dot{\varphi}(t)+F(\varphi(t))=0$ 
on $(0,\infty)$ with $\varphi(0+)=\infty$, cf. 
\autoref{lem:Fprop}. In some examples, we can compute the relaxation function 
$\varphi$  explicitly. For example, $\varphi(t) = -\log \big(\tanh 
t\big)$ for the unweighted two-point graph. 

\medskip

The differential Harnack inequality \eqref{eq:diffHarnack} implies pointwise 
bounds on the function $u$ itself by a chaining argument, cf. \cite{LiYa86} and 
\cite{BHLLMY15} in the case of graphs. We apply the same strategy. In the case 
of finite graphs, our Harnack inequality then reads as follows:

\begin{theorem} \label{theo:Harnack_ineq_finite}
Let $G$ be a finite graph satisfying the assumption of 
\autoref{theo:finite-graphs}. Assume $u:[0,\infty)\times V\rightarrow 
(0,\infty)$ is a solution of the heat equation
on $G$ and $0<t_1<t_2$ and $x_1, x_2\in V$. Then 

\begin{align} \label{eq:Harnack_ineq_finite}
u(t_1,x_1)\le u(t_2,x_2) \exp\Big(\int_{t_1}^{t_2} 
\varphi(t)\,dt \Big) \exp\Big(\frac{2\mu_{max} 
d(x_1,x_2)^2}{w_{min} (t_2-t_1)} \Big)\,.
\end{align}
\end{theorem}

\begin{remark} 
For the sake of this introduction, we choose to 
present our results, \autoref{theo:finite-graphs} and 
\autoref{theo:Harnack_ineq_finite}, for the case of finite graphs. Versions for 
general locally finite connected graphs are given in 
\autoref{sec:li-yau-infinite} and \autoref{sec:harnack}.
\end{remark}

\begin{remark} 
Note that in all examples studied in this work, the relaxation function 
$\varphi$ turns out to be integrable at $t=0$. Thus, it is possible to consider 
the case $t_1=0$ in \eqref{eq:Harnack_ineq_finite}. This is in contrast to the 
Harnack inequality on manifolds. 
\end{remark}

\medskip

Let us comment on related results in the literature. One approach to Li-Yau 
type estimates on graphs is given in 
\cite{BHLLMY15} and several subsequent
works. Since the results of the present work are closely related, let us 
explain the approach of \cite{BHLLMY15}. The authors establish the following 
estimate
\begin{align}\label{eq:harvard-diffHarnack}
\frac{\Gamma(\sqrt{u})(t,x)}{u(t,x)}- \frac{\partial_t 
(\sqrt{u})(t,x)}{\sqrt{u}(t,x)}
\leq \frac{n}{2t} \quad (t > 0, x \in V)
\end{align}
for positive solutions $u$ to the heat equation on $G$, which should be 
contrasted with \eqref{eq:diffHarnack-classical} and \eqref{eq:diffHarnack}. A 
significant difference between this result and our estimate is that we estimate 
the term $\Psi_\Upsilon(\log u)$, which, in some sense, is the 
correct discrete replacement for $|\nabla \log u|^2.$ As a
consequence of \eqref{eq:harvard-diffHarnack}, the authors obtain a 
Harnack inequality 
\begin{align}\label{eq:harvard-harnack-Zd}
u(t_1,x_1) \leq u(t_2,x_2) \left(\frac{t_2}{t_1} \right)^n \exp \left(\frac{4D 
\rho^2(x_1,x_2)}{t_2-t_1}\right)
\quad (0 < t_1 < t_2, x_i \in V),
\end{align}
where $D$ equals the maximal degree of a vertex in $G$. Note that 
$\mathbb{Z}^d$ 
satisfies CDE($n$,0) with $n=2d$. Thus, the exponent $n$ in 
$\left(\frac{t_2}{t_1} \right)^n$ is off by a factor $4$ from what one would 
expect, based on the corresponding estimates in the Euclidean space. In 
\cite{BHLLMY15}, the authors study graphs which satisfy the exponential 
curvature dimension inequality CDE($n$,0). 

\medskip

Computation rules and estimates for the logarithm of positive 
solutions appear also in \cite{Mue14}. The main aim of \cite{Mue14} is to 
establish generalized curvature dimension inequalities and to prove a Li-Yau 
inequality on finite graphs. In this way, \cite{Mue14} enhances some of the 
results of 
\cite{BHLLMY15}, e.g., the estimate \eqref{eq:harvard-harnack-Zd}. The 
relation between the conditions  (curvature dimension inequalities) of 
\cite{Mue14} and \cite{BHLLMY15} is studied in \cite{Mue17}.

\medskip

The main difference between the present work and the approach in 
\cite{BHLLMY15}, \cite{Mue14} and other existing works is that we do not 
restrict ourselves to expressions resp. functions of the form $t \mapsto 
c t^{-1}$ in the differential Harnack 
inequality. In this respect, \eqref{eq:diffHarnack} 
and \eqref{eq:harvard-diffHarnack} are rather different. As can be seen from 
\eqref{eq:Harnack_ineq_finite}, the function $\varphi$ plays an important role 
in the pointwise estimate for the positive solution $u$. In light of 
\eqref{eq:Harnack-classical} the estimate \eqref{eq:harvard-harnack-Zd} looks 
natural but the behavior for $t_1 \to 0+$ seems far from being optimal. Note 
that the 
Laplace operator, when defined on a graph with bounded degree, is a bounded 
operator. Thus, one 
should expect a robust estimate for all $t_1 > 0$. We believe that an optimal 
result requires the time-dependence to be captured by a function $\varphi$ 
depending on the graph under consideration. This is why, in our 
approach, $\varphi$ is linked to the graph via the CD-function $F$ from the 
condition CD($F$;0). \\
Another difference between the present work and \cite{BHLLMY15} concerns the 
analysis on infinite graphs. Infinite graphs are not studied in \cite{Mue14}. 
As in the case of Riemannian manifolds, it is necessary to decompose 
$\Psi_\Upsilon(\log u)$ into two parts in order to apply successfully cut-off 
functions. We develop a systematic approach for this procedure, which we call 
$\alpha$-calculus, where $\alpha \in [0,1)$. The special case $\alpha = 
\frac12$ is strongly connected to the methods of \cite{BHLLMY15}.

\medskip

It is worth mentioning that, in general, Ricci curvature bounds play an 
important role. If the Ricci curvature of a Riemannian manifold is 
bounded from below by a strictly positive number, then, in addition to the 
Harnack inequality, several properties can be established. Isoperimetric 
inequalities follow as well as lower bounds for the eigenvalues of the 
Laplacian. There have been several attempts to develop a notion of Ricci  
curvature bounds  for  discrete or, more generally, for non-smooth spaces 
starting from the theory of Bakry and Emery \cite{BaEm85}, which is based 
on properties of the corresponding semigroup. For recent developments in this 
direction, we refer to \cite{HLLY14}, \cite{JoLi14},  
\cite{HJL15}, \cite{LaMc16}, \cite{CLP16}, and \cite{KKRT16}. Note that the 
last mentioned work contains several concrete examples and 
computations. Following the theory of Lott, Villani, and Sturm for metric 
measure spaces, techniques from optimal transport have been applied, cf. 
\cite{Oll09}, \cite{BoSt09}, \cite{Maa11}, \cite{ErMa12}, \cite{Mie13}, 
\cite{EMP15}, or the nice survey in \cite{Oll10}. Since, in the present work, 
neither semigroups nor optimal transport are used, we omit a further discussion 
here.

\medskip

The article is organized as follows: In \autoref{sec:fund-id} we study 
computation rules for difference operators, in particular a discrete version of 
the chain rule. It 
turns out, that it is possible to obtain nice formulas for expressions of the 
form of $\Delta(\log u)$. In \autoref{sec:curv-ineq} we introduce a new notion 
of curvature inequality, which is parametrized by a CD-function $F$. This 
function is computed explicitly for several examples of graphs in 
\autoref{sec:curv-ineq}. \autoref{sec:li-yau-finite} contains the proof of the 
Li-Yau 
estimate on finite graphs. In \autoref{sec:li-yau-infinite} we explain 
how to obtain a similar result on infinite Ricci-flat graphs. In the 
special case of the lattice $\Z$ resp. the sequence $(\tau \Z)_{\tau>0}$, we 
show in \autoref{sec:example-Z} how to recover the classical 
sharp Li-Yau estimate on $\R$ in the limit $\tau \to 0+$. Finally, in 
\autoref{sec:harnack} we apply the chaining argument from 
\cite{BHLLMY15} and derive a Harnack inequality from the Li-Yau estimate. We 
prove the result for locally finite graphs thus establishing 
\autoref{theo:Harnack_ineq_finite}.

\section{Fundamental identities}\label{sec:fund-id}

This section is concerned with a basic identity, which can be viewed as a kind 
of chain rule for the operator $\Delta$. We refer to it as the {\em 
fundamental identity}. Given a function $H:\iR\to \iR$, we also define the 
operator $\Psi_H:\iR^V\to \iR^V $ by
\begin{align} \label{def:Psi}
\Psi_H (u)(x)=\,\frac{1}{\mu(x)}\sum_{y\sim x}w_{xy} H\big(u(y)-u(x)\big),\quad 
x\in V,\;u\in \iR^V.
\end{align}
Observe that in case of the function $H(y)=\frac{1}{2}y^2$ we have $\Psi_H( 
u)=\Gamma(u)$.

\begin{lemma} \label{lem:firstFI}
Let $\Omega\subset \iR$ be an open set and $u\in \iR^V$ such that the range of 
$u$ is contained in $\Omega$.
Let further $H\in C^1(\Omega;\iR)$. Then there holds
\begin{align} \label{FI1}
\Delta\big( H(u(x)\big)=H'(u(x)) \Delta u(x)+\frac{1}{\mu(x)}\sum_{y\sim x} 
w_{xy} \Lambda_H\big(u(y),u(x)\big),\quad x\in V,
\end{align}
where we set
\begin{align}\label{eq:def_Lambda}
\Lambda_H(w,z):=H(w)-H(z)-H'(z)(w-z),\quad w,z\in \iR.
\end{align}
\end{lemma}

\begin{proof}
For each neighbor $y$ of $x$ we have
\begin{align}
H(u(y))-&H(u(x)) =H'(u(x))\big( u(y)-u(x)\big) \nonumber\\
&+\Big(H(u(y))-H(u(x))-H'(u(x))\big[
u(y)-u(x)\big]\Big).   \label{fund1}
\end{align}
Multiplying \eqref{fund1} by the weight $w_{xy}/\mu(x)$ and summing over all  
$y\sim x$ yields the assertion.
\end{proof}

Note that the quantity $\Lambda_H(w,z)$ resembles the Bregman distance from 
convex analysis. Identity \eqref{FI1} is the analogue in the graph setting of 
the classical 
rule
\[
\Delta H(u)=H'(u)\Delta u+  H''(u)|\nabla u|^2 \qquad (u \in C^2(\iR^d))\,. 
\]
See also \cite{Zac13} for an 
application of a similar identity in the context of evolution 
equations with fractional time derivatives. Note that in case of a convex 
function $H$ we obtain $\Lambda_H\ge 0$ and thus identity \eqref{FI1} yields the
inequality $\Delta H(u)\ge H'(u) \Delta u$. Let us look at some examples.

\begin{example}\label{exa:fund-id_square}
Take $\Omega=\iR$ and $H(y)=\frac{1}{2}y^2$. Then
\[
\Lambda_H(w,z)=\frac{1}{2}w^2-\frac{1}{2}z^2-z(w-z)= \frac{1}{2}(w-z)^2
\]
and thus we get for any $u\in \iR^V$ and $x\in V$
\begin{align*}
\frac{1}{2}\Delta(u^2)(x) &=u(x)\Delta u(x)+\frac{1}{2\mu(x)}\,\sum_{y\sim x} 
w_{xy} \big(u(y)-u(x)\big)^2.
\end{align*}
Hence
\begin{align} \label{gammaform}
\Delta(u^2) =2u\Delta u+2\Gamma(u).
\end{align}
\end{example}

\begin{example}\label{exa:fund-id_square-root}
Take $\Omega=(0,\infty)$ and consider the function $H(y)=\sqrt{y}$, $y>0$. Then
\begin{align*}
\Lambda_H(w,z)& =\sqrt{w}-\sqrt{z}-\frac{1}{2\sqrt{z}}\,(w-z)
=\,-\,\frac{1}{2\sqrt{z}}\big(\sqrt{w}-\sqrt{z}\big)^2.
\end{align*}
Assuming that $u\in \iR^V$ is positive, the fundamental identity then gives
\[
\Delta(\sqrt{u})(x)=
\frac{1}{2\sqrt{u(x)}}\,\Delta u(x)-\frac{1}{2\mu(x)\sqrt{u(x)}}
\sum_{y\sim x} w_{xy}\big(\sqrt{u(y)}-\sqrt{u(x)}\big)^2.
\]
Multiplying by $2\sqrt{u}$ we obtain
\begin{align} \label{sqrtformula}
2\sqrt{u}\Delta \sqrt{u}=\Delta u-2\Gamma(\sqrt{u}).
\end{align}
Relation \eqref{sqrtformula} is the key identity for the square root approach 
used in \cite{BHLLMY15}. Observe that \eqref{sqrtformula} is also an immediate 
consequence
of formula \eqref{gammaform}; just substitute $v=\sqrt{u}$ in 
\eqref{sqrtformula} to see this.
\end{example}

\begin{example}\label{exa:fund-id_log}
Take  $\Omega=(0,\infty)$ and $H(y)=-\log y$, $y>0$. Then
\begin{align*}
\Lambda_H(w,z)& =-\log w+\log z+\frac{1}{z}\,(w-z)\\
& = \log \big(\frac{z}{w}\big)+\frac{w}{z}-1\\
& = \Upsilon(\log w-\log z),
\end{align*}
where
\[
\Upsilon(y):=e^y-1-y=\sum_{j=2}^\infty \frac{y^j}{j!},\quad y\in \iR.
\]
Assuming that $u\in \iR^V$ is positive, the fundamental identity yields
\[
-\Delta(\log u)(x)=-\frac{1}{u(x)}\Delta u(x)
+\frac{1}{\mu(x)}\sum_{y\sim x}w_{xy} 
\Upsilon\big(\log u(y)-\log u(x)\big).
\]
This shows the important relation
\begin{align} \label{palmerel}
\frac{1}{u}\Delta u=\Delta (\log u)+\Psi_\Upsilon (\log u),
\end{align}
which is remarkable since the right-hand side is formulated using only terms 
involving the function $\log u$. Replacing the positive function $u$ in 
\eqref{palmerel} by $u^\alpha$ with 
$\alpha>0$ yields the identity
\begin{align} \label{palmalpha}
\frac{\Delta( u^\alpha)}{\alpha u^\alpha}=\Delta (\log 
u)+\frac{1}{\alpha}\,\Psi_{\Upsilon_\alpha} (\log u),
\end{align}
where we set $\Upsilon_\alpha(y)=\Upsilon(\alpha y)$.
\end{example}

\begin{lemma} \label{alphalemma}
Let $\alpha\in (0,1)$. The function $g_\alpha:\iR \to \iR$ defined by
\[
g_\alpha(z)=\Upsilon(z)-\frac{1}{\alpha}\,\Upsilon(\alpha z),\quad z\in \iR,
\]
is nonnegative on $\iR$ and satisfies
\begin{align} \label{quadest}
g_\alpha(z)\ge \frac{1-\alpha}{2}\,z^2,\quad z\ge 0.
\end{align}
Moreover, we have the representation
\[
g_\alpha(z)=h_\alpha\big(e^z\big),\quad z\in \iR,
\]
where
\[
h_\alpha(z)=z-\frac{1}{\alpha}z^\alpha+\frac{1-\alpha}{\alpha},\quad z\ge 0.
\]
In particular, in case $\alpha=\frac{1}{2}$, there holds 
\[
g_{1/2}(z)=\big(e^{z/2}-1\big)^2,\quad z\in \iR.
\]
\end{lemma}

\begin{proof} 
By definition of $\Upsilon$ we have
\[
g_\alpha(z)=e^z-1-z-\frac{1}{\alpha}\big(e^{\alpha z}-1-\alpha 
z\big)=e^z-\frac{1}{\alpha}e^{\alpha z}
+\frac{1}{\alpha}-1=h_\alpha\big(e^z\big),
\]
and thus
\[
g_\alpha'(z)=e^z-e^{\alpha z},
\]
which shows that $g_\alpha$ is strictly decreasing on $(-\infty,0]$ and 
strictly 
increasing on $[0,\infty)$,
with $g_\alpha(0)=0$ being the global minimum. 
For any $z>0$, Taylor's theorem gives
\begin{align} \label{Taylor1}
g_\alpha(z)=\,\frac{1}{2}\,g_\alpha''(\xi)z^2 
\end{align}
with some $\xi\in (0,z)$. Clearly, the function $g_\alpha''(z)=e^z-\alpha 
e^{\alpha z}$ 
is strictly
increasing on $[0,\infty)$ and $g_\alpha''(0)=1-\alpha$, and so \eqref{Taylor1} 
implies 
the inequality \eqref{quadest}. The last assertion follows from the identity
\[
h_{1/2}(z)=\big(\sqrt{z}-1\big)^2,\quad z\ge 0.
\]
\end{proof}

\medskip

Note that \autoref{alphalemma} also shows that in case $\alpha=1/2$ we have for 
any positive
$u\in \iR^V$ and $x\in V$ that
\begin{align}
\Psi_\Upsilon(\log u)(x)-\frac{1}{\alpha}\Psi_{\Upsilon_\alpha}(\log u)(x) 
& =\Psi_{(\exp(\cdot/2)-1)^2}(\log u)(x) \nonumber\\
& = \,\frac{1}{\mu(x)}\sum_{y\sim x}w_{xy} \Big( e^{(\log u(y)-\log 
u(x))/2}-1\Big)^2\nonumber\\
&  = \,\frac{1}{\mu(x)}\sum_{y\sim x}w_{xy} \Big( 
\frac{\sqrt{u(y)}}{\sqrt{u(x)}}-1\Big)^2\nonumber\\
&  = \frac{2\Gamma(\sqrt{u})(x)}{u(x)}. \label{Gammaroot}
\end{align} 

The aforementioned computation rules directly apply to more general nonlocal 
operators. We formulate this result for the Euclidean space.

\begin{lemma}
 Assume $(\mu(x, \d y))_{x \in \R^d}$ is a family of measures on the Borel sets 
of $\R^d$ satisfying $\mu(x,\{x\}) = 0$ and 
\begin{align*}
 \sup\limits_{x \in \R^d} \int_{\R^d} \big( 1 \wedge |x-y|^2 \big) \mu(x,\d y) 
< \infty \,.
\end{align*}
Assume $H \in C^2(\R^d)$, $x \in \R^d$ and $u: \R^d \to \R$ are such that
\begin{align*}
\mathcal{L}u(x) := \lim\limits_{\varepsilon \to 0} \int_{\R^d \setminus 
B_{\varepsilon}} \big( u(y) - u(x) \big) \mu(x, \d y)  
\end{align*}
and $\mathcal{L}(H\circ u)(x)$ exist. Then
\begin{align} \label{eq:fund-id_nonlocal}
\mathcal{L}\big( H \circ u\big) (x) = H'(u(x)) \mathcal{L} u(x) + 
\lim\limits_{\varepsilon \to 0} \int_{\R^d \setminus B_{\varepsilon}} \Big( 
\Lambda_H\big(u(y),u(x)\big) \Big) \mu(x,\d y) \,,
\end{align}
with $\Lambda_H$ as in \eqref{eq:def_Lambda}.
\end{lemma}

The proof of this result is as simple as the proof of 
\autoref{lem:firstFI}. Repeating \autoref{exa:fund-id_square}, 
\autoref{exa:fund-id_square-root} and \autoref{exa:fund-id_log} for the 
case of the fractional Laplace operator $\Delta^s = 
-(-\Delta)^s$ with $0 < s < 1$ in $\R^d$, we obtain the following 
computation rules for sufficiently regular functions $u$:

\begin{align}
 \Delta^s(u^2) &= 2u\Delta u+2\Gamma(u)\,, \\
 \Delta^s \sqrt{u} &= \frac{\Delta^s 
u}{2\sqrt{u}} - \frac{2 \Gamma(\sqrt{u})}{2\sqrt{u}} \,, \\
 \Delta^s \log u &= \frac{\Delta^s u}{u}  - 
 c_{d,s} \int_{\R^d} \frac{\Upsilon \big(\log u (y) - \log u 
(\cdot))}{|y-\cdot|^{d+2s}} \d y \,.
\end{align}
Here, $\Gamma$ denotes the carr\'e du champ operator that corresponds with 
$\Delta^s$, and $\Upsilon$ is as above. The constant $c_{d,s}$ is the 
normalizing constant that appears in the representation of $\Delta^s$ as an 
integrodifferential operator. It satisfies $c_{d,s} \asymp (1-s)s$ for $0 
< s < 1$. For sufficiently regular functions $v$, the following observation 
holds: 
\begin{align*}
  c_{d,s} \int_{\R^d} \frac{\Upsilon \big(v(y) - v (\cdot))}{|y-\cdot|^{d+2s}} 
\d y  \to |\nabla v|^2  \quad \text{ as } s \to 1-\,.
\end{align*}

\section{Conditions related to curvature-dimension 
inequalities}\label{sec:curv-ineq}

In this section, we introduce a family of conditions CD$_\alpha$($F$;0) on 
graphs. Here $\alpha \in [0,1)$ is a parameter and $F:\,[0,\infty)\rightarrow 
[0,\infty)$ is a function with $F(0)=0$ and some additional properties. As we 
will show, condition CD$_0$($F$;0)\footnote{In the sequel, we will write 
CD($F$;0) instead of CD$_0$($F$;0).} ensures that positive solutions to the 
heat equation satisfy a Li-Yau type estimate. We provide examples of 
graphs that satisfy CD$_0$($F$;0) and examples that do not have this property. 
The case $\alpha \in (0,1)$ is of particular interest for infinite graphs, cf. 
\autoref{sec:li-yau-infinite}.

\subsection{A new version of the CD-inequality}
\begin{definition}\label{def:F}
A continuous function $F:\,[0,\infty)\rightarrow [0,\infty)$ is called 
CD-function, if $F(0)=0$, $F(x)/x$ is strictly
increasing on $(0,\infty)$,
and $1/F$ is integrable at $\infty$. 
\end{definition}

Note that for any CD-function $F$ we have $F(x)>0$ for all $x\in (0,\infty)$ 
and 
$F$ is strictly increasing
on $[0,\infty)$. An example of a CD-function is given by $F(x)=c x^2$ with 
$c>0$.

\begin{proposition} \label{prop-perm1}
If $F_1, F_2$ are CD-functions, then the functions $F_1+F_2$ and 
$\min(F_1,F_2)$ 
are CD-functions, as well. In addition, 
$\alpha F_1(\beta\cdot)$ is a CD-function 
for every $\alpha,\beta\in (0,\infty)$.
\end{proposition}

\begin{proof}
The argument for $F_1+F_2$ is straightforward. As to the minimum
$F:=\min(F_1,F_2)$, note that 
\[
H(x):=\frac{F(x)}{x}=\min(H_1(x),H_2(x))\quad \mbox{with}\; 
H_i(x):=\frac{F_i(x)}{x},\quad x\in (0,\infty).
\]
It follows from the intermediate value theorem that the minimum of two strictly 
increasing and continuous
functions is again strictly increasing. Thus $H$ is strictly increasing on 
$(0,\infty)$. Further, we have for $x_1,x_2>0$ that
\[
\frac{1}{\min(x_1,x_2)}< \frac{1}{x_1}+\frac{1}{x_2} ,
\]
and so it is evident that the integrability of $1/F_i$ at $\infty$, $i=1,2$, 
implies the same property for $1/F$.
This shows that $F$ is a CD-function. The last assertion is 
obvious.
\end{proof}

\begin{remark} \label{suffcond}
Let $g:\,[0,\infty)\to [0,\infty)$ be a strictly convex function with $g(0)=0$. 
Then the function $g(x)/x$ 
is strictly increasing on $(0,\infty)$. In fact, strict convexity of $g$ 
implies that the difference quotients
of $g$ are strictly increasing, and thus 
$g(x)/x=(g(x)-g(0))/x$ is strictly increasing. Note that a CD-function need not 
be convex as the example
$F(x)=\min(x^2,x^3)$ shows.
\end{remark}

The following family of CD-functions plays a central role in the context of 
Ricci-flat graphs.

\begin{proposition} \label{CDFexample}
Let $\lambda\in (0,1)$ and the function $F:\,[0,\infty)\rightarrow \iR$ be 
defined by 
\begin{align} \label{lambdaform}
F(x)=e^{-\frac{1-\lambda}{2} x}\Big(\lambda 
e^{(1-\lambda)x}+(1-\lambda)e^{-\lambda x}-1\Big),\quad x\ge 0.
\end{align}
Then $F$ is a strictly convex CD-function. Moreover, the function
$F(x)/x$ is convex in $[4,\infty)$ and 
\begin{align} \label{gminprop}
\frac{d}{dx}\,\Big(\frac{F(x)}{x}\Big)\ge \frac{1}{2}\,\lambda(1-\lambda) 
e^{-2(1+\lambda)},\quad x\in (0,4].
\end{align}
\end{proposition}

\begin{proof} Let $S(x)$ denote the term in brackets in 
\eqref{lambdaform} and 
set
$\beta=\frac{1-\lambda}{2}$. By the convexity of the
exponential function, we have
\begin{align*}
1=e^0=e^{\lambda[(1-\lambda)x]+(1-\lambda)[-\lambda x]}\le \lambda 
e^{(1-\lambda)x}+
(1-\lambda)e^{-\lambda x}.
\end{align*}
This shows non-negativity of $S$, and thus $F(x)\ge 0$ for all $x\ge 0$. 
Evidently, $F(0)=0$. Further,
\begin{align}
F''(x) & =e^{-\beta x}\big( \beta^2 S(x)-2\beta S'(x)+S''(x)\big)\nonumber\\
& = e^{-\beta x}\Big(\beta^2 S(x)
-2\beta \big[\lambda(1-\lambda)\big(e^{(1-\lambda)x}-e^{-\lambda 
x}\big)\big]\nonumber\\
&\;\;\;+\big[\lambda(1-\lambda)^2e^{(1-\lambda)x}+(1-\lambda)\lambda^2 
e^{-\lambda x}\big]\Big)\nonumber\\
& = e^{-\beta x}\big(\beta^2 S(x)+\lambda(1-\lambda)e^{-\lambda 
x}\big)\nonumber\\
& \ge \lambda(1-\lambda)e^{-(\beta+\lambda) x}. \label{secondderiv}
\end{align}
This implies strict convexity of $F$, and thus by \autoref{suffcond} that 
$F(x)/x$ is strictly increasing
on $(0,\infty)$. Since $F(x)$ is exponentially increasing as $x\to \infty$, 
$1/F$ is integrable at $\infty$.
Hence $F$ is a CD-function. 

As to \eqref{gminprop}, we have for $x>0$
\begin{align*}
\frac{d}{dx}\,\Big(\frac{F(x)}{x}\Big)=\frac{xF'(x)-F(x)}{x^2},
\end{align*}
and by Taylor's theorem
\[
0=F(0)=F(x)+F'(x)(-x)+\frac{1}{2}F''(\xi)x^2,
\] 
for some $\xi\in(0,x)$. Using \eqref{secondderiv}, it follows that for $x\in 
(0,4]$
\begin{align*}
\frac{d}{dx}\,\Big(\frac{F(x)}{x}\Big) & =\frac{1}{2}F''(\xi)\ge 
\frac{1}{2}\,\lambda(1-\lambda)e^{-(\beta+\lambda) \xi}\\
& \ge \frac{1}{2}\,\lambda(1-\lambda)e^{-(\frac{1-\lambda}{2}+\lambda) x}\ge 
\frac{1}{2}\,\lambda(1-\lambda) e^{-2(1+\lambda)}.
\end{align*}

Turning to the convexity of $F(x)/x$, we have for $x>0$
\begin{align*}
\frac{d^2}{dx^2}\,& \Big(\frac{F(x)}{x}\Big)=\frac{x^2 
F''(x)-2xF'(x)+2F(x)}{x^3}\\
& = \frac{e^{-\beta x}}{x^3}\Big( x^2(\beta^2 
S(x)+\lambda(1-\lambda)e^{-\lambda 
x}\big)
-2x\big(-\beta S(x)+S'(x)\big)+2S(x)\Big)\\
& \ge \frac{e^{-\beta x}}{x^3}\Big( \frac{(1-\lambda)^2}{4} x^2 \lambda 
e^{(1-\lambda)x}-2x \frac{1-\lambda}{2}
\lambda(1-\lambda) e^{(1-\lambda)x}\Big)\\
& = \frac{e^{-\beta x}}{4x^2}\,(x-4)\lambda(1-\lambda)^2 e^{(1-\lambda)x},
\end{align*}
and thus $(F(x)/x)''\ge 0$ for all $x\in [4,\infty)$. 
\end{proof}

\begin{lemma} \label{lem:Fprop}
Let $F:[0,\infty)\rightarrow [0,\infty)$ be a CD-function. Then there is a 
unique strictly positive solution $\varphi$ of the ODE 
\begin{align} \label{eq:phiODE}
\dot{\varphi}(t)+F(\varphi(t))=0,\quad t>0,
\end{align}
with $\varphi(0+)=\infty$. The function 
$\varphi$ is strictly decreasing and log-convex, and it 
satisfies $\varphi(\infty)=0$. 
\end{lemma}

\begin{proof} We define $G(x)=\int_x^\infty dr/F(r), x>0$. Then 
$G'(x)=-1/F(x)<0$, that is,
$G$ is strictly decreasing. Since $F(x)/x$ is
increasing on $(0,\infty)$, we have $F(x)\le F(1)x$ for all $x\in (0,1]$, and 
thus
\[
G(x)=\int_x^1 \frac{dr}{F(r)}\,dr+\int_1^\infty\frac{dr}{F(r)}\,dr \ge 
\,\frac{1}{F(1)}\,\int_x^1 \frac{dr}{r}\,dr+G(1),\quad x\in (0,1],
\]
which shows that $G(0+)=\infty$. Observe also that $G(\infty)=0$. 

Suppose $\varphi$ is a strictly positive solution of the ODE \eqref{eq:phiODE} 
on 
$(0,\infty)$ with
$\varphi(0+)=\infty$. Then for $t, t_1\in (0,\infty)$ we have
\[
t-t_1=\int_{t}^{t_1}  
\frac{\dot{\varphi}(\tau)}{F(\varphi(\tau))}\,d\tau=\int_{\varphi(t)}^{
\varphi(t_1)} \frac{dr}{F(r)}.
\]
Sending $t_1\to 0+$ yields $t=G(\varphi(t))$, that is 
\begin{align} \label{odesol}
\varphi(t)=G^{-1}(t),\quad t>0,
\end{align} 
which shows uniqueness. On the other hand, it is easy to verify that 
\eqref{odesol} defines 
a strictly positive solution $\varphi$ of the ODE \eqref{eq:phiODE} with 
$(0,\infty)$ as its maximal interval 
of existence. Evidently, $\varphi(0+)=\infty$, $\varphi(\infty)=0$, and 
$\dot{\varphi}(t)<0$  for all  $t\in (0,\infty)$.

Finally, since $\varphi$ is strictly decreasing and $F(x)/x$ is strictly 
increasing, the function
\[
\eta(t):=\frac{d}{dt}\,\big(\log 
\varphi(t)\big)=\frac{\dot{\varphi}(t)}{\varphi(t)}
=-\frac{F(\varphi(t))}{\varphi(t)},\quad t>0,
\]
is strictly increasing, which in turn implies that $\varphi$ is log-convex. 
\end{proof}

\begin{definition}\label{def:relax-fct}
Let $F:[0,\infty)\rightarrow [0,\infty)$ be a CD-function. The positive 
function 
$\varphi$ that solves \eqref{eq:phiODE}
with $(0,\infty)$ as maximal interval of existence is called relaxation 
function associated with $F$. 
\end{definition}

We now discuss the asymptotic properties of the relaxation function. Here and 
in the sequel, we write $f(r) \sim g(r) \quad (r \to a)$ for $a \in \{0, 
+\infty\}$ and two functions $f$ and $g$, if the ratio $f(r)/g(r)$ stays 
bounded for $r \to a$. Note that we use the same symbol to describe that two 
vertices $x,y \in V$ are neighbors, i.e., $x \sim y$.

\begin{lemma} \label{lem:asymptrelax}
Let $F$ be a CD-function and $\varphi$ the corresponding relaxation function. 
Then the following statements hold.
\begin{itemize}
\item [(i)] Let $\tilde{x}\in [0,\infty)$ and 
$\tilde{F}:\,[\tilde{x},\infty)\to 
(0,\infty)$ be continuous.
Assume further that $F(r)\sim \tilde{F}(r)$ as $r\to \infty$ and define 
$\tilde{G}:\,[\tilde{x},\infty)
\to (0,a]$ with $a=\int_{\tilde{x}}^\infty dr/\tilde{F}(r)$ by 
$\tilde{G}(x)=\int_{\tilde{x}}^\infty dr/\tilde{F}(r)$, $x\ge \tilde{x}$. Let 
$\tilde{\varphi}:\,(0,a]\to (0,\infty)$ be defined by
$\tilde{\varphi}(t)=\tilde{G}^{-1}(t)$. Then
\[
\tilde{\varphi}(t) \sim \varphi(t)\quad \mbox{as}\;t\to 0+.
\]
In particular, if $F(r)\sim c\, e^{\gamma r}$ as $r\to \infty$ with $c,\gamma 
>0$ the relaxation function has a logarithmic
singularity at $0+$, 
\[
{\varphi}(t)\sim -\frac{1}{\gamma}\,\log t\quad \mbox{as}\;t\to 0+.
\]
\item[(ii)] Suppose that $F(r)\sim \nu r^2$ as $r\to 0+$ with some constant 
$\nu>0$, and assume that there exists
$\nu_0>0$ such that
$F(r)\ge \nu_0 r^2$ for all $r\ge 0$. Then 
\[
\varphi(t) \sim \frac{1}{\nu t}\quad \mbox{as}\;t\to \infty.
\]
\end{itemize}

\end{lemma}
\begin{proof}
(i) The first assertion follows directly from the representation formula for 
$\varphi$,
\[
\varphi(t)=G^{-1}(t)\quad \mbox{with}\;\;G(x)=\int_{x}^\infty \frac{dr}{{F}(r)}.
\]
Recall that as $t\to 0+$ we have that $\varphi(t)\to \infty$ and thus the 
formula for $G$ shows that the behavior of $F$ at
$\infty$ determines the behavior of $\varphi$ at $0+$. In the case 
$\tilde{F}(r)=c\, e^{\gamma r}$ we find that
\[
\tilde{G}(x)=\frac{1}{c}\,\int_{{x}}^\infty e^{-\gamma 
r}\,dr=\frac{1}{c\gamma}\, e^{-\gamma x},\quad x\ge 0,
\]
which yields
\[
\tilde{\varphi}(t)=-\frac{1}{\gamma}\log\big(c\gamma t \big) \sim 
-\frac{1}{\gamma}\,\log t\quad \mbox{as}\;t\to 0+.
\]

(ii) Let $F_0(r)=\nu r^2$, $r\ge 0$. We set 
\[
F_\tau(r)=\frac{1}{\tau^2}\,F(\tau r),\quad \tau>0,\,r\ge 0, 
\]
and
\[
G_\tau(x)=\int_{x}^\infty \frac{dr}{{F_\tau}(r)},\quad \tau\ge 0,\,x>0.
\]
We first claim that $F_\tau \to F_0$ uniformly on any interval $(0,r_1]$ as 
$\tau\to 0+$. In fact, letting $r_1>0$
the assumptions on
$F$ imply that given $\varepsilon>0$ there is $\delta>0$ such that $F(s)/(\nu 
s^2)\le 1+\varepsilon/(\nu r_1^2)$ for all $s\in (0,\delta]$.
Suppose now that $\tau\in (0,\delta/r_1]$. Then we have for $r\in (0,r_1]$ that 
$\tau r\le \delta$ and thus
\begin{align*}
|F_\tau(r)-F_0(r)| & \le \nu r^2\Big| \frac{F(\tau r)}{\nu (\tau r)^2}-1\Big|
\le \varepsilon.
\end{align*}

Next, it follows from the previous property and the lower bound for $F$ that 
$G_\tau \to G_0$ uniformly on any interval $[x_0, x_1]\subset (0,\infty)$ as 
$\tau\to 0+$. This can be seen by writing
\[
G_\tau(x)=\int_1^\infty \frac{dr}{F_\tau(r)}+\int_x^1 
\frac{dr}{F_\tau(r)},\quad 
x\in [x_0,x_1],
\]
where the convergence of the first integral to $\int_1^\infty 
\frac{dr}{F_0(r)}$ 
follows from the
dominated convergence theorem. 

Using the property that $G_\tau \to G_0$ on any compact subinterval of 
$(0,\infty)$ as $\tau\to 0+$ it
is not difficult to check that then 
for each $t\in (0,\infty)$ we have
\begin{equation} \label{phitauconv}
\varphi_\tau(t):=G_\tau^{-1}(t)\to  G_0^{-1}(t)= \frac{1}{\nu t}=:\varphi_0(t)
\end{equation}
as $\tau\to 0+$. Observe that by the definitions of $F_\tau$ and $G_\tau$,
\[
\dot{\varphi}_\tau(t)=-F_\tau\big(\varphi_\tau(t)\big)=-\frac{1}{\tau^2} 
F\big(\tau  \varphi_\tau(t)\big),
\quad t\in (0,\infty), 
\]  
as well as $\varphi_\tau(0+)=\infty$. Invoking \autoref{lem:Fprop}, this shows 
that 
\[
\varphi_\tau(t)=\frac{1}{\tau}\,\varphi\big(\frac{t}{\tau}\big),\quad t,\tau>0,
\]
which together with \eqref{phitauconv} gives for any fixed $t>0$
\[
\frac{t}{\tau}\,\varphi\big(\frac{t}{\tau}\big)= t \varphi_\tau(t) \to 
\frac{1}{\nu}\quad 
\mbox{as}\;\tau\to 0+.
\]
Hence $s\varphi(s)\to 1/\nu$ as $s\to \infty$. This proves (ii).
\end{proof}

For $\alpha\in [0,1)$, $v\in \iR^V$ and $x\in V$ we define (with $L=-\Delta$)
\begin{align} 
{\mathcal L}_0(v)(x) & = Lv(x),\nonumber\\
{\mathcal L}_\alpha(v)(x) & =-\frac{1}{\alpha}\,\Psi_{\Upsilon'}(\alpha v)(x), 
\quad \mbox{if}\;\alpha
\in(0,1) \label{Lalpha}
\end{align}
and
\begin{align} \label{Calpha}
{\mathcal C}_\alpha(v)(x) :=  \frac{1}{\mu(x)}\sum_{y\sim x} w_{xy} 
e^{\alpha(v(y)-v(x))}
 \big(\Psi_{\Upsilon'}(v)(y)-\Psi_{\Upsilon'}(v)(x)\big).
\end{align}
Observe that
\[
{\mathcal C}_0(v)(x)=\Delta \Psi_{\Upsilon'}(v)(x)
\]
and
\[
{\mathcal L}_\alpha(v)(x)\to Lv(x)\quad \mbox{as}\;\alpha\to 0+.
\]
Notice as well that
\[
{\mathcal L}_\alpha(v)(x)=Lv(x)-\frac{1}{\alpha}\,\Psi_{\Upsilon}(\alpha v)(x),
\]
which in particular shows that positivity of ${\mathcal L}_\alpha(v)(x)$ 
implies the same property for
$Lv(x)$. Note $\mathcal L_\alpha (\log u)$ equals $\Delta^\psi u$ of 
\cite{Mue14} for the choice $\psi(s) = (s^\alpha - 1)/\alpha$. The following 
condition is of great importance throughout this paper.

\begin{definition}\label{def:CD}
Let $\alpha\in [0,1)$, $F$ be a CD-function, $G=(V,E)$ an undirected graph, and 
$x\in V$. 
We say that the graph 
$G$ satisfies condition CD$_\alpha$($F$;0) at $x \in V$ for the generalized 
Laplace operator $\Delta$ given by \eqref{genlaplace}, if 
for every function $v:V \to \R$ satisfying
\[
{\mathcal L}_\alpha(v)(x)>0 \quad \mbox{and}\quad {\mathcal L}_\alpha(v)(x)\ge 
{\mathcal L}_\alpha(v)(y)\;\;\mbox{for all}\;y\sim x,
\]
there holds
\begin{align} \label{CDUngleichung}
{\mathcal C}_\alpha(v)(x)\ge F\big(Lv(x)\big).
\end{align}
We say that $G$ satisfies CD$_\alpha$($F$;0) if it satisfies CD$_\alpha$($F$;0) 
at every $x\in V$. In the case $\alpha=0$, we drop the subscript '0' in 
the notation and simply speak of the CD-inequality CD($F$;0).
\end{definition}

\begin{remark} (i) The notion CD$_\alpha$($F$;0) suggests that there is 
a more general condition CD$_\alpha$($F$;$K$), where $K \in \R$ denotes some 
curvature bound. So far, we do not allow for terms that measure the curvature as 
it is the case in the classical curvature dimension inequality. This will be 
subject to further research. (ii) The condition CD$_0$($F$;0) relates to the 
classical curvature 
dimension inequality in a natural way. Note that, in the case $\alpha 
=0$, $\mathcal{L}_\alpha$ equals $-\Delta$. Now, let us look at the Euclidean 
case. 
Assume $x \in \R^d$ and $v: \R^d \to \R$ is a smooth function such that the 
function $-\Delta v$ has a local, strictly positive  maximum in $x$. Then 
\begin{align}\label{eq:classical-cd_comparison}
\underset{\geq 0}{\underbrace{\Delta \Delta v(x)}} + 2 
\underset{=0}{\underbrace{(\nabla \Delta v(x), \nabla v(x))}} + 2 
\sum_{i,j=1}^{d} 
(\partial_i \partial_j v(x))^2 \geq \tfrac{2}{d} (-\Delta v(x))^2 \,, 
\end{align}
which is the classical curvature dimension inequality. Note that the 
left-hand side of \eqref{eq:classical-cd_comparison} corresponds to $\Delta 
\Psi_{\Upsilon'}(v)(x)$. In this sense, the condition CD$_0$($F$;0) 
is consistent with the classical curvature dimension inequality. 
\end{remark}

Using \autoref{prop-perm1} we immediately obtain the following.

\begin{proposition} \label{minCD}
Let $\alpha\in[0,1)$, $G=(V,E)$ be a graph, $F_i$ be CD-functions for 
$i=1,\ldots,l$ 
and assume  that for any $x\in V$
the graph satisfies CD$_\alpha$($F_i$,0) at $x$ for some 
$i\in \{1,\ldots,l\}$. Set $F:=\min(F_1,\ldots,F_l)$. Then the graph 
satisfies CD$_\alpha$($F$;0).
\end{proposition}

\subsection{Some simple illustrating examples}

\begin{example} \label{exa:twopoints}
We first consider the connected graph that only consists of two
different vertices, say $x_1$ and $x_2$. 
For the Laplace operator, we take the most simple form (without weight),
that is,
\[
\Delta u(x)=u(\tilde{x})-u(x),\quad x\in V=\{x_1,x_2\},
\]
where $\tilde{x}_1=x_2$ and vice versa. Let $v\in \iR^V$ and $x\in V$. Then we 
have
\begin{align*}
\Delta \Psi_{\Upsilon'}(v)(x) & = \Psi_{\Upsilon'}(v)(\tilde{x}) - 
\Psi_{\Upsilon'}(v)(x)\\
& = \Upsilon'\big(v(x)-v(\tilde{x}) \big)-\Upsilon'\big(v(\tilde{x})-v(x) 
\big)\\
& =  e^{Lv(x)}-e^{-Lv(x)}\\
& = F\big(Lv(x)\big),
\end{align*}
where $F(a)=2\sinh a$, which is easily verified to be a CD-function. 
Thus, condition CD($2 \sinh$;0) is satisfied. A straight-forward 
computation shows that
the relaxation function corresponding to $F$ is given by
\begin{align} \label{phi2vertex}
\varphi(t)=\log\Big(\frac{1+e^{-2t}}{1-e^{-2t}}\Big)=-\log \big(\tanh 
t\big),\quad t>0.
\end{align}
In the case $\alpha\in (0,1)$ one obtains
\begin{equation} \label{twopointalpha}
{\mathcal C}_\alpha(v)(x)=e^{-\alpha Lv(x)}\big(e^{Lv(x)}-e^{-Lv(x)}\big)
\ge e^{(1-\alpha)Lv(x)}-e^{-(1-\alpha)Lv(x)},
\end{equation}
that is, the CD-inequality CD$_\alpha$($F_\alpha$;0) holds with
\[
F_\alpha(y)=2\sinh\big((1-\alpha) y\big).
\]
Note that $\tilde{F}(y)=e^{-\alpha y}\big(e^y-e^{-y}\big)$, $y\ge 0$ is not a 
CD-function, since
$F(y)/y$ is decreasing near $0$. Note also that in \eqref{twopointalpha} we 
used 
that $Lv(x)=v(x)-v(\tilde{x})>0$, which follows from ${\mathcal 
L}_\alpha(v)(x)>0$.
\end{example}

\begin{example} \label{exa:triangle}{\ }

\begin{minipage}{0.7\textwidth}
We next consider the case of a triangle, i.e., $V=\{x_*,x_1,x_2\}$ and 
$E=\{x_*x_1,x_*x_2,
x_1x_2\}$. Again, we look at the most simple case without weights and with 
$\mu\equiv 1$. Let $v\in \iR^V$ and set $z_j=v(x_j)$  for $j\in\{*,1,2\}$ and 
$a_j=z_*-z_j$ for $j\in\{1,2\}$. 
\end{minipage}
\begin{minipage}{0.3\textwidth}
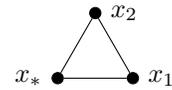
\begin{figure}[H]
\begin{tikzpicture}[main_node/.style={circle,fill=black,draw,minimum 
size=0.2em,inner sep=1.5pt]}]

    \node[main_node] [label={[distance=0.1cm]180:$x_{\ast}$}](1) at (0,0) {};
    \node[main_node] [label={[distance=0.1cm]00:$x_1$}](2) at (1, 0) {};
    \node[main_node] [label={[distance=0.1cm]00:$x_2$}](3) at (0.5,0.866025) {};

    \draw (1) -- (2) -- (3) -- (1);
\end{tikzpicture}
\caption{Triangle}
\label{fig:triangle} 
\end{figure}
\end{minipage}

Now, 
\begin{align*}
\Delta \Psi_{\Upsilon'}(v)(x_*) & = 
\Psi_{\Upsilon'}(v)(x_1)+\Psi_{\Upsilon'}(v)(x_2) - 2\Psi_{\Upsilon'}(v)(x_*)\\
& = 
\Upsilon'(z_*-z_1)+\Upsilon'(z_2-z_1)+\Upsilon'(z_*-z_2)+\Upsilon'(z_1-z_2)\\
& \;\;\;-2\big(\Upsilon'(z_1-z_*)+\Upsilon'(z_2-z_*)   \big)\\
& = e^{a_1}+e^{a_1-a_2}+e^{a_2}+e^{a_2-a_1}-2e^{-a_1}-2e^{-a_2}=:f(a_1,a_2).
\end{align*}
Further,
\begin{align*}
Lv(x_*)& =(z_*-z_1)+(z_*-z_2)= a_1+a_2,\\
Lv(x_1) & = (z_1-z_*)+(z_1-z_2)=a_2-2a_1,\\
Lv(x_2) & = (z_2-z_*)+(z_2-z_1)=a_1-2a_2.
\end{align*}
We see that $Lv$ has a positive maximum at $x_*$ if and only if $a_j\ge 0$ for 
$j=1,2$ and $a_1+a_2>0$.
Assuming this, by symmetry, we may assume without loss of generality that $0\le 
a_1\le a_2$. 
Then
\[
\frac{\partial f}{\partial a_2}=-e^{a_1-a_2}+e^{a_2}+e^{a_2-a_1}+2e^{-a_2}>0,
\]
and thus $f(a_1,a_2)\ge f(a_1,a_1)$, which in turn yields
\begin{align*}
\Delta \Psi_{\Upsilon'}(v)(x_*) & \ge F\big(Lv(x_*)\big),
\end{align*}
where
\[
F(a)=2\Big(e^{\frac{a}{2}}+1-2e^{-\frac{a}{2}}\Big).
\]
Observe that $F$ is not a CD-function, since $F(a)=3a-a^2/2+O(a^3)$ as $a\to 
0+$, which implies that 
$F(x)/x$ is not increasing near $0$. However, one can find many CD-functions 
$\tilde{F}$ with $F\ge \tilde{F}$
on $[0,\infty)$, e.g.\ $\tilde{F}(a)=4\sinh(a/2)$,
and so CD($\tilde{F}$;0) holds for any such function.
\end{example}

The aforementioned examples are special cases of the class of complete graphs. 
Next, let us treat complete graphs in general.

\begin{example}
Let $G=(V,E)$ be a complete graph with $D+1$ vertices, $D \in \N$. That is, for 
every pair of 
vertices $x,y\in V$ with
$x\neq y$ we have $x\sim y$. Let $V=\{x_0,x_1,\ldots,x_D\}$ and $\alpha\in 
[0,\frac{1}{2}]$.
We consider the case without edge weights and with $\mu(y)=\mu_0>0$ for all 
$y\in V$.
Suppose that $v:\,V\rightarrow \iR$ is such that
\[
\mathcal{L}_\alpha(v)(x_0)>0\quad \mbox{and} \quad 
\mathcal{L}_\alpha(v)(x_0)\ge \mathcal{L}_\alpha(v)(x_j)
\quad \mbox{for all}\;j=1,\ldots,D. 
\]
Setting
\[
F_\alpha(a):=\frac{D}{\mu_0^2}e^{-\frac{\alpha\mu_0}{D}a}\big(e^{\frac{\mu_0}{D}
a}-1\big)
\big( De^{-\frac{\mu_0}{D}a}+1\big),\quad a\ge 0,
\]
we claim that
\begin{equation} \label{estimatecompletegraphs}
\mathcal{C}_\alpha(v)(x_0)\ge F_\alpha\big(Lv(x_0)\big).
\end{equation}
Indeed, putting $z_k=v(x_k)$ for $k=0,1,\ldots, D$ and $\psi(a)=e^{-\alpha 
a}(e^a-1)$, $a\in \iR$, we have
\begin{align*}
\mathcal{C}_\alpha(v)(x_0) & =\frac{1}{\mu_0} \sum_{k=1}^D e^{\alpha(z_k-z_0)} 
\Big(\Psi_{\Upsilon'}(v)(x_k)-
\Psi_{\Upsilon'}(v)(x_0)\Big)\\
& = \frac{1}{\mu_0^2} \sum_{k=1}^D e^{\alpha(z_k-z_0)} \Big( \sum_{j=0,\,j\neq 
k}^D e^{z_j-z_k}
-\sum_{j=1}^D e^{z_j-z_0}\Big)\\
& =  \frac{1}{\mu_0^2} \sum_{k=1}^D e^{\alpha(z_k-z_0)} 
\Big(\sum_{j=1}^D \big(e^{z_j-z_k}-e^{z_j-z_0}\big)+e^{z_0-z_k}-1\Big)\\
& =  \frac{1}{\mu_0^2} \sum_{k=1}^D e^{-\alpha(z_0-z_k)} 
\big(e^{z_0-z_k}-1\big) \Big(\sum_{j=1}^D e^{z_j-z_0}+1\Big)\\
& \ge  \frac{D}{\mu_0^2}\,\psi\big(\frac{1}{D}\,\sum_{k=1}^D 
(z_0-z_k)\big)\Big(\sum_{j=1}^D e^{z_j-z_0}+1\Big)\\
& =  \frac{D}{\mu_0^2}\,\psi\big( \frac{\mu_0}{D} Lv(x_0) 
\big)\Big(\sum_{j=1}^D 
e^{z_j-z_0}+1\Big),
\end{align*}
by convexity of $\psi$ on $\iR$. Since $\psi$ is positive on $(0,\infty)$ and 
$Lv(x_0)\ge \mathcal{L}_\alpha(v)(x_0)>0$,
the $\psi$-term in the last line is positive. Using the convexity of the 
exponential function we may thus further deduce
that
\begin{align*}
\mathcal{C}_\alpha(v)(x_0) &  \ge \frac{D}{\mu_0^2}\,\psi\big( \frac{\mu_0}{D} 
Lv(x_0) \big)
\Big(D \exp\big(\frac{1}{D}\sum_{j=1}^D (z_j-z_0)\big)+1\Big)\\
& = \frac{D}{\mu_0^2}\,\psi\big( \frac{\mu_0}{D} Lv(x_0) \big)
\Big(D \exp\big(-\frac{\mu_0}{D}Lv(x_0)\big)+1\Big)=F_\alpha\big(Lv(x_0)\big).
\end{align*} 
Note that $F_\alpha$ is not a CD-function in general. Note that in the case 
$\alpha =0, \mu_0 = 1$, we obtain
\begin{align*}
F_0(a)=D \big(e^{\frac{a}{D}} - De^{\frac{-a}{D}} + (D-1)\big)\,,
\end{align*}
which reduces to $F_0(a)= e^{a} - e^{-a}= 2 \sinh(a)$ in the case 
$D=1$ and to $F_0(a)=2 \big(e^{\frac{a}{2}} - 2e^{\frac{-a}{2}} + 1\big)$ 
in the case of $D=2$. Thus, the case of general complete graphs is consistent 
with \autoref{exa:twopoints} and \autoref{exa:triangle}.  
\end{example}

\begin{example} \label{exa:chain}{\ }

\medskip

\begin{minipage}{0.65\textwidth}
The next example is a path consisting of three vertices. Let 
$V=\{x_*,x_1,x_2\}$ and $E=\{x_*x_1,x_*x_2\}$.
We consider the case without weights and with $\mu(x_i)=1$, $i=1,2$ and 
$\mu(x_*)=2$, so $\mu$ coincides
at every vertex with its degree. Letting $v\in \iR^V$ we use the same notation 
as in \autoref{exa:triangle}.
\end{minipage}
\begin{minipage}{0.35\textwidth}
\begin{figure}[H]
\centering
\begin{tikzpicture}[main_node/.style={circle,fill=black,draw,minimum 
size=0.2em,inner sep=1.5pt]}]
    \node[main_node] [label={[distance=0.2cm]90:$x_1$}](1) at (0,0) {};
    \node[main_node] [label={[distance=0.2cm]90:$x_{\ast}$}](2) at (1,0) {};
    \node[main_node] [label={[distance=0.2cm]90:$x_2$}](3) at (2,0) {};
    \draw (1) -- (2) -- (3);
\end{tikzpicture}
\label{fig:chain}
\caption{A chain-like graph}
\end{figure}
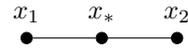
\end{minipage}

\medskip

Then, we have for the vertex $x_*$
\begin{align*}
\Delta \Psi_{\Upsilon'}(v)(x_*) & = 
\frac{1}{2}\Big(\Psi_{\Upsilon'}(v)(x_1)+\Psi_{\Upsilon'}(v)(x_2) - 
2\Psi_{\Upsilon'}(v)(x_*)\Big)\\
& = \frac{1}{2}\Big(\Upsilon'(z_*-z_1)+\Upsilon'(z_*-z_2)- 2\cdot \frac{1}{2} 
\big(\Upsilon'(z_1-z_*)
+\Upsilon'(z_2-z_*) \big)\Big)\\
& =  \frac{1}{2}\Big(e^{a_1}+e^{a_2}-e^{-a_1}-e^{-a_2}\Big)=:\tilde{f}(a_1,a_2),
\end{align*}
and 
\[
Lv(x_*)= \frac{1}{2}(a_1+a_2),\quad Lv(x_1)=-a_1,\quad Lv(x_2)=-a_2.
\]
$Lv$ has a positive maximum at $x_*$ if and only if $3a_1+a_2\ge 0$, 
$3a_2+a_1\ge 0$ and $a_1+a_2>0$. 
Assuming this, by symmetry, we may assume that $a_1\le a_2$. Then $3a_1+a_2\ge 
0$ implies that $3a_2+a_1\ge 0$,
so the first condition is the stronger one and will be assumed. 

Suppose first that $a_1<0$. The function $\tilde{f}$ is strictly increasing 
w.r.t.\ $a_2$,
and thus $\tilde{f}(a_1,a_2)\ge \tilde{f}(a_1,-3a_1)$. This leads to
\[
\Delta \Psi_{\Upsilon'}(v)(x_*)\ge F_1\big(Lv(x_*)\big)\quad \mbox{with}\;\;
F_1(a)= \frac{1}{2}\Big(e^{3a}+e^{-a}-e^{a}-e^{-3a}\Big),
\]
since for $a_2=-3a_1$ we have $Lv(x_*)=-a_1>0$.

Next, suppose that $a_1> 0$. Then $\tilde{f}(a_1,a_2)\ge \tilde{f}(a_1,a_1)$, 
which yields
\[
\Delta \Psi_{\Upsilon'}(v)(x_*)\ge F_2\big(Lv(x_*)\big)\quad \mbox{with}\;\;
F_2(a)= e^{a}-e^{-a}=2\sinh a.
\]
Note that here $Lv(x_*)=a_1>0$. The case $a_1=0$ leads to the function $F_2$ as 
well.

One can show that $F_1(a)\ge F_2(a)$ for all $a\ge 0$. 
Hence CD($2\sinh$;0) holds at $x_*$. Note that here, we work with the same 
CD-function as in \autoref{exa:twopoints}.

Let us now study an endpoint of the path. At the vertex $x_1$, we have
\begin{align*}
\Delta \Psi_{\Upsilon'}(v)(x_1) & = 
\Psi_{\Upsilon'}(v)(x_*)-\Psi_{\Upsilon'}(v)(x_1)\\
& = \frac{1}{2} \big(\Upsilon'(z_1-z_*)
+\Upsilon'(z_2-z_*)\big)-\Psi_{\Upsilon'}(v)(z_*-z_1)\\
& = \frac{1}{2} \big( e^{-a_1}+e^{-a_2}\big)-e^{a_1}=:\hat{f}(a_1,a_2).
\end{align*}
The condition $Lv(x_1)\ge Lv(x_*)$ is equivalent to
$3a_1+a_2\le 0$, and $Lv(x_1)>0$ means that $a_1=-Lv(x_1)<0$.
Since $\hat{f}$ is strictly decreasing w.r.t.\ $a_2$, we obtain 
$\hat{f}(a_1,a_2)\ge \hat{f}(a_1,-3a_1)$ (by increasing
$a_2$ for fixed $a_1<0$). This gives
\[
\Delta \Psi_{\Upsilon'}(v)(x_1)\ge F_3\big(Lv(x_*)\big)\quad \mbox{with}\;\;
F_3(a)=  \frac{1}{2} \big(e^{a}+e^{-3a}\big)-e^{-a}.
\]
Observing that
\[
F_3(a)= \frac{1}{2} \,e^{-a} \big(e^a-e^{-a}\big)^2\quad 
\mbox{and}\;\;F_3''(a)=F_3(a)+4e^{-3a}>0
\]
we easily see that $F_3$ is a CD-function. Hence, for $i=1,2$, the 
condition CD($F_3$;0) holds at $x_i$.

Concerning the entire graph, it follows from \autoref{minCD} that the 
CD($F$;0) holds with $F=\min(F_2,F_3)=F_3$.
\end{example}

\subsection{Ricci-flat graphs}\label{subsec:Ricci-flat}

Next, we show that Ricci-flat graphs satisfy the condition CD($F$;0) with a 
CD-function $F$ that we can compute explicitly. The notion of Ricci-flat 
graphs has been introduced in \cite{ChYa96} as a notion of graphs with 
nonnegative curvature. 

\begin{definition} \label{defRicci}
Let $G=(V,E)$ be a $D$-regular graph with $D\in \iN$, let $x \in V$ and $N(x) 
= \{x\} \cup \{y \in V | y \sim x\}$ . $G$ is called Ricci-flat at $x$ 
if, there exist maps $\eta_1,\ldots,\eta_D: N(x) \rightarrow V$ 
such that the following conditions are satisfied:
\begin{itemize}
\item[(i)] $\eta_i(y) \sim y$ for all $i\in \{1,\ldots,D\}$ and all $y \in 
N(x)$.
\item[(ii)] $\eta_i(y) \neq \eta_j(y)$ for $y \in 
N(x)$ and $i \neq j$.
\item[(iii)] $\bigcup_{j=1}^D \eta_i(\eta_j(x))=\bigcup_{j=1}^D 
\eta_j(\eta_i(x))$ for all $i\in \{1,\ldots,D\}$.
\end{itemize}
The graph $G$ is called Ricci-flat if it is Ricci-flat at every vertex $x 
\in V$.
\end{definition}

The graph $\Z^d$ with $x \sim y \Leftrightarrow |x-y|_1=1$ is Ricci-flat. Any
Cayley graph of a finitely generated group is Ricci-flat if the generating 
system is closed under conjugation. 

\begin{minipage}
{0.95\textwidth} 
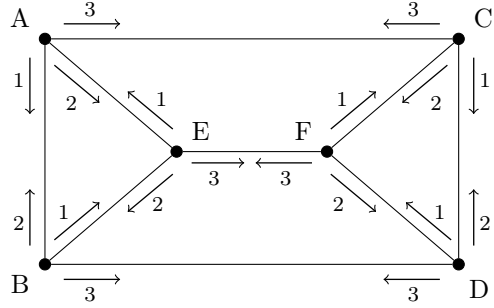
\begin{figure}[H]
\begin{tikzpicture}[main_node/.style={circle,fill=black,draw,minimum 
size=0.2em,inner sep=1.5pt]}]

    \node[main_node] [label={[distance=0.2cm]160:A}] (A) at (0,0) {};
    \node[main_node] [label={[distance=0.2cm]20:C}] (B) at (5.5, 0)  {};
    \node[main_node] [label={[distance=0.2cm]280:D}] (C) at (5.5, -3) {};
    \node[main_node] [label={[distance=0.2cm]190:B}] (D) at (0, -3)  {};
    \node[main_node] [label={[distance=0.2cm]20:E}] (E) at (1.75, -1.5) {};
    \node[main_node] [label={[distance=0.2cm]160:F}] (F) at (3.75, -1.5) {};

 
    \draw[arrows=->, line width=0.5pt] (-0.2, -0.25) -- (-0.2, -1);
		\draw (-0.35,-0.55) node {\footnotesize{1}};    
	
	\draw[arrows=->, line width=0.5pt] (5.7, -0.25) -- (5.7, -1);
		\draw (5.85,-0.55) node {\footnotesize{1}}; 

    \draw[arrows=<-, line width=0.5pt] (-0.2, -2) -- (-0.2, -2.75);
         \draw (-0.35,-2.45) node {\footnotesize{2}};         
		
    \draw[arrows=<-, line width=0.5pt] (5.7, -2) -- (5.7, -2.75);
         \draw (5.85,-2.45) node {\footnotesize{2}};

	\draw[arrows=->, line width=0.5pt] (0.25, 0.2) -- (1, 0.2);
		\draw (0.6,0.4) node {\footnotesize{3}}; 	
		
	\draw[arrows=<-, line width=0.5pt] (4.5, 0.2) -- (5.25, 0.2);
		\draw (4.9,0.4) node {\footnotesize{3}}; 	

	\draw[arrows=<-, line width=0.5pt] (2.65, -1.65) -- (1.95, -1.65);
		\draw (2.25,-1.85) node {\footnotesize{3}}; 	
		    
	\draw[arrows=<-, line width=0.5pt] (2.8, -1.65) -- (3.55, -1.65);
		\draw (3.2,-1.85) node {\footnotesize{3}}; 	

    \draw[arrows=->, line width=0.5pt] (0.25, -3.2) -- (1, -3.2);
		\draw (0.6,-3.4) node {\footnotesize{3}}; 	
		    
	\draw[arrows=<-, line width=0.5pt] (4.5, -3.2) -- (5.25, -3.2);
		\draw (4.9,-3.4) node {\footnotesize{3}};


    \draw[arrows=->, line width=0.5pt] (1.7, -1.2) -- (1.1, -0.7); 
    	\draw (1.55, -0.8) node {\footnotesize{1}};

    \draw[arrows=<-, line width=0.5pt] (0.725, -2.175) -- (0.125, -2.675); 
    	\draw (0.25, -2.3) node {\footnotesize{1}};	
    
    \draw[arrows=->, line width=0.5pt] (5.4, -2.675) -- (4.8, -2.175); 
    	\draw (5.25, -2.3) node {\footnotesize{1}};
    	
	\draw[arrows=<-, line width=0.5pt] (4.4, -0.7) -- (3.8, -1.2); 
    	\draw (3.95, -0.8) node {\footnotesize{1}};
    	
    \draw[arrows=->, line width=0.5pt] (3.8, -1.8) -- (4.4, -2.3); 
    	\draw (3.9, -2.2) node {\footnotesize{2}};	
    
    \draw[arrows=->, line width=0.5pt] (1.7, -1.8) -- (1.1, -2.3); 
    	\draw (1.5, -2.2) node {\footnotesize{2}};	
    
    \draw[arrows=<-, line width=0.5pt] (0.725, -0.85) -- (0.125, -0.35); 
    	\draw (0.35, -0.85) node {\footnotesize{2}};	
    	
  	\draw[arrows=<-, line width=0.5pt] (4.75, -0.85) -- (5.35, -0.35); 
    	\draw (5.2, -0.85) node {\footnotesize{2}};

    \draw (A) -- (B) -- (C) -- (D) -- (A) -- (E) -- (D) (E) -- (F) -- (B) (F) 
-- 
(C);
    
\end{tikzpicture}
\caption{A Ricci flat graph together with the maps $\eta_i$, $i \in 
\{1,2,3\}$.}
\label{fig:S3}
\end{figure}
\end{minipage}

\medskip

It is proved in \cite{LiYa10} that Ricci-flat graphs satisfy $\Gamma_2(f,f) 
\geq 0$ for every $f \in V \to \R$, thus it is reasonable to think of 
Ricci-flat graphs as graphs with nonnegative curvature. One can also think of 
Ricci-flat graphs as generalizations of Cayley graphs of Abelian groups. We 
will make use of the following property of Ricci-flat graphs, 
which is proved in \cite{Mue14}.

\begin{lemma} \label{lem-Ricflat}
Let $G=(V,E)$ be a $D$-regular graph which is Ricci-flat at the vertex $x\in 
V$. 
Let $\eta_1,\ldots,\eta_D$ be the maps as in \autoref{defRicci}.
\begin{itemize}
\item[(i)] For any function $u: V\rightarrow \iR$ and for all $i\in 
\{1,\ldots,D\}$ one has
\begin{align} \label{sumproperty}
\sum_{j=1}^D u(\eta_i(\eta_j(x)))=\sum_{j=1}^D u(\eta_j(\eta_i(x))).
 \end{align}
 \item[(ii)] For every $i\in \{1,\ldots,D\}$ there exists a unique $i^*\in 
\{1,\ldots,D\}$ such that $\eta_{i}(\eta_{i^*}(x))=x$.
 Moreover, the map $i\mapsto i^*$ is a permutation of $\{1,\ldots,D\}$. 
\end{itemize}
\end{lemma}

As mentioned above, we can show that Ricci-flat graphs satisfy the condition 
CD($F$;0) for a function $F$ that we can compute explicitly. This is 
the 
content of the next result. Note that, in several examples, it is possible to 
prove the condition CD($F$;0) with a larger function $\widetilde{F}$, i.e., the 
function $F$ given in \autoref{theo:CDRiccimain} is not best possible.

\begin{theorem} \label{theo:CDRiccimain}
Let $G=(V,E)$ be a $D$-regular unweighted Ricci-flat graph with $D\ge 2$. 
Assume 
that $\mu(y)=\mu_0>0$ for all $y\in V$. Then CD($F$;0) holds 
with
\begin{align} \label{RicciF}
F(a)=\frac{D}{\mu_0^2} \exp\big(-\frac{\mu_0}{D}a\big)\Big[ 
\Upsilon\Big(\frac{2\mu_0}{D}a \Big)
+(D-1) \Upsilon\Big(- \frac{2\mu_0}{D(D-1)}a \Big)\Big].
\end{align}
\end{theorem}

\begin{proof} We first verify that $F$ is a CD-function. Setting 
$\eta=\frac{\mu_0}{D}$ and 
$\lambda=\frac{1}{D}$ we can
write
\begin{align*}
F(a) & =\frac{1}{\mu_0\eta}\,e^{-\eta a}\Big[e^{2\eta 
a}+(D-1)e^{-\frac{2\eta}{D-1}a}-D\Big]\\
& = \frac{1}{\eta^2}\,e^{-\eta a}\Big[\lambda e^{2\eta 
a}+(1-\lambda)e^{-\frac{2\eta}{D-1}a}-1\Big].
\end{align*}
We scale the argument by putting 
$\tilde{a}=\frac{2\eta}{1-\lambda}a=\frac{2\mu_0}{D-1}a$ and
introduce the function $\tilde{F}$ by means of $\tilde{F}(\tilde{a})=F(a)$. 
This gives
\begin{align*}
\tilde{F}(\tilde{a})=\frac{1}{\eta^2}\,e^{-\frac{1-\lambda}{2} \tilde{a}}
\Big[\lambda e^{(1-\lambda) \tilde{a}}+(1-\lambda)e^{-\lambda \tilde{a}}-1\Big].
\end{align*}
\autoref{CDFexample} and \autoref{prop-perm1} now imply that $\tilde{F}$, and 
thus also $F$, are
strictly convex CD-functions.

Let now $x\in V$ and $v\in \iR^V$ such that 
\[
Lv(x)>0 \quad \mbox{and}\quad Lv(x)\ge Lv(\eta_j(x))\;\;\mbox{for 
all}\;j=1,\ldots, D.
\]
Set $z=v(x)$, $z_i=v(\eta_i(x))$ and $z_{ij}=v(\eta_j(\eta_i(x)))$ for 
$i,j=1,\ldots, D$. We have
\begin{align}
\Delta \Psi_{\Upsilon'}(v)(x) & = \frac{1}{\mu_0}\,\sum_{i=1}^D 
\big[\Psi_{\Upsilon'}(v)(\eta_i(x))
-\Psi_{\Upsilon'}(v)(x)\big] \nonumber \\
 & = \frac{1}{\mu_0^2}\,\sum_{i=1}^D \sum_{j=1}^D \big[ 
\Upsilon'(z_{ij}-z_i)-\Upsilon'(z_j-z)\big]
 = \frac{1}{\mu_0^2}\,\sum_{i=1}^D \sum_{j=1}^D \big[ 
e^{z_{ij}-z_i}-e^{z_j-z}\big] \nonumber\\
 & = \frac{1}{\mu_0^2}\,\sum_{j=1}^D e^{z_j-z} \sum_{i=1}^D 
\big(e^{z_{ij}-z_i-z_j+z}-1\big). \label{psiform}
\end{align}
Setting $w_j=z-\frac{1}{2}z_j-\frac{1}{2}z_{j^*}$ and recalling that $z_{ 
j^{*}j}=z$, the inner sum can be written as
\begin{align*}
\sum_{i=1}^D \big(e^{z_{ji}-z_i-z_j+z}-1\big)=e^{2w_j}-1+\sum_{i=1, i\neq 
j^*}^D 
\big(e^{z_{ij}-z_i-z_j+z}-1\big). 
\end{align*}
By convexity of the exponential function, \autoref{lem-Ricflat} (i) and the 
local maximum property of $Lv$ at $x$ we may now estimate as follows.
\begin{align*}
\sum_{i=1, i\neq j^*}^D  &\big(e^{z_{ij}-z_i-z_j+z}-1\big)  \ge 
(D-1) \Big[ \exp\Big (\frac{1}{D-1}\sum_{i=1, i\neq j^*}^D 
[z_{ij}-z_i-z_j+z]\Big)-1\Big]\\
& = (D-1) \Big[ \exp\Big (\frac{1}{D-1}\sum_{i=1}^D 
[z_{ij}-z_i-z_j+z]-\frac{2w_j}{D-1}\Big)-1\Big]\\
& =  (D-1) \Big[ \exp\Big (\frac{1}{D-1}\sum_{i=1}^D 
[z-z_i+z_{ji}-z_j]-\frac{2w_j}{D-1}\Big)-1\Big]\\
& = (D-1) \Big[ \exp\Big 
(\frac{\mu_0}{D-1}[Lv(x)-Lv(\eta_j(x)]-\frac{2w_j}{D-1}\Big)-1\Big]\\
& \ge  (D-1) \Big[ \exp\Big (-\frac{2w_j}{D-1}\Big)-1\Big].
\end{align*}
Combining this and the previous identities yields
\begin{align} \label{Ricflatstep}
\Delta \Psi_{\Upsilon'}(v)(x)\ge  \frac{1}{\mu_0^2}\,\sum_{j=1}^D e^{z_j-z} 
\Big(e^{2w_j}-1+ (D-1) \Big[ \exp\Big (-\frac{2w_j}{D-1}\Big)-1\Big]   \Big).
\end{align}

The next step consists in symmetrizing the sum. Since we do not have 
$(j^*)^*=j$ 
in general,
we use the rearrangement inequality, which says that for all permutations 
$\pi$ on
$\{1,\ldots,D\}$ and all $0\le a_1\le a_2\le \ldots \le a_D$ and all $0\le 
b_1\le b_2\le \ldots \le b_D$, one 
has
\begin{align} \label{Rea}
\sum_{j=1}^D a_{\pi(j)} b_j\ge \sum_{j=1}^D a_{D+1-j} b_j.
\end{align}

Without restriction of generality, we may assume that $z_1\le z_2\le \ldots \le 
z_D$. We set
$j':=D+1-j$ and $\tilde{w}_j=z-\frac{1}{2}z_j-\frac{1}{2}z_{j'}$. Using 
\eqref{Rea} we then have
\begin{align*}
\sum_{j=1}^D e^{z_j-z} \exp\Big (-\frac{2w_j}{D-1}\Big)
&  =\sum_{j=1}^D \exp\Big( z_j-z+\frac{1}{D-1}(z_j-2z)\Big) 
\exp\Big(\frac{1}{D-1}z_{j^*}\Big)\\
 & \ge \sum_{j=1}^D \exp\Big( z_j-z+\frac{1}{D-1}(z_j-2z)\Big) 
\exp\Big(\frac{1}{D-1}z_{j'}\Big)\\
&  = \sum_{j=1}^D e^{z_j-z}\exp\Big (-\frac{2\tilde{w}_j}{D-1}\Big).
\end{align*}
Furthermore,
\[
\sum_{j=1}^D e^{z_j-z} 
e^{2w_j}= \sum_{j=1}^D e^{z-z_{j^*}}= \sum_{j=1}^D e^{z-z_{j'}}=\sum_{j=1}^D 
e^{z_j-z} 
e^{2\tilde{w}_j}.
\]
These relations and \eqref{Ricflatstep} imply that
\begin{align} \label{Ricflatstep2}
\Delta \Psi_{\Upsilon'}(v)(x)\ge  \frac{1}{\mu_0^2}\,\sum_{j=1}^D e^{z_j-z} 
\Big(e^{2\tilde{w}_j}-1+ (D-1) \Big[ \exp\Big 
(-\frac{2\tilde{w}_j}{D-1}\Big)-1\Big]   \Big).
\end{align}

Compared to \eqref{Ricflatstep}, inequality \eqref{Ricflatstep2} has the 
advantage that 
$\tilde{w}_j=\tilde{w}_{j'}$, since $(j')'=j$. Employing this and the convexity 
of the exponential function
and $F$ we have
\begin{align*}
\Delta \Psi_{\Upsilon'}(v)(x) & \ge  \frac{1}{2\mu_0^2}\,\sum_{j=1}^D 
\big[e^{z_j-z}+e^{z_{j'}-z}\big] 
\Big(e^{2\tilde{w}_j}-1+ (D-1) \Big[ \exp\Big 
(-\frac{2\tilde{w}_j}{D-1}\Big)-1\Big]   \Big) \\
& \ge  \frac{1}{\mu_0^2}\,\sum_{j=1}^D  e^{-\tilde{w}_j}  
\Big(e^{2\tilde{w}_j}-1+ (D-1) \Big[ \exp\Big 
(-\frac{2\tilde{w}_j}{D-1}\Big)-1\Big]   \Big) \\
& = \frac{1}{D}\,\sum_{j=1}^D F\Big(\frac{D\tilde{w}_j}{\mu_0}\Big)
\ge F\Big(\sum_{j=1}^D \frac{\tilde{w}_j}{\mu_0}\Big)=F\big(Lv(x)\big).
\end{align*}
This proves the asserted inequality. 
\end{proof}

\medskip

\noindent It turns out that for Ricci-flat graphs with constant $\mu$, in 
general, the CD-function $F$ provided by \autoref{theo:CDRiccimain} is optimal, 
at least if $D$ is an even number. This can be seen by looking at the lattice 
$\iZ^d$.
More precisely, we have the following result.

\begin{theorem} Let $G=(V,E)$ be the lattice $\iZ^d$ and consider the case 
without weights and with the Laplace operator given by
\[
\Delta u(x)=\,\frac{1}{\mu_0}\,\sum_{y\sim x} \big(u(y)-u(x)\big),\quad x\in 
\iZ^d,
\]
where $\mu_0>0$ is a constant. Then for any $a>0$, there exists a function 
$v\in 
\iR^V$ satisfying
\[
Lv(0)=a>0 \quad \mbox{and}\quad Lv(0)\ge Lv(y)\;\;\mbox{for all}\;y\sim 0,
\]
and such that
\[
\Delta \Psi_{\Upsilon'}(v)(0)= F\big(Lv(0)\big)=F(a),
\]
where $F$ is the CD-function given by \eqref{RicciF} with $D=2d$. 
\end{theorem}

\begin{proof}
Let $e_j$ be the $j$th unit vector in $\iR^d$ and set $\eta_j(x)=x+e_j$ and 
$\eta_{j+d}(x)=x-e_j$ for $j=1,\ldots,d$ and $x\in \iZ^d$. For any vertex, the 
mapping $j\to j^*$ from
\autoref{lem-Ricflat}(ii) is then given by $j^*=j+d$ for $j=1,\ldots,d$
and $j^*=j-d$ for $j=d+1,\ldots,2d$. 

Let $\alpha>\beta$ and define
$v(0)=\alpha$ and $v(\eta_j(0))=\beta$ for $j=1,\ldots,2d$. For $i,j\in 
\{1,\ldots,2d\}$ with $i\neq j^*$ 
we further set $v(\eta_j(\eta_i(0))=\gamma$. We put $v(x)=0$ elsewhere. We then 
have
\[
Lv(0)=\frac{2d(\alpha-\beta)}{\mu_0}>0.
\]
The idea is now to choose $\gamma\in \iR$ such that $Lv(\eta_j(0))=Lv(0)$ for 
all $j=1,\ldots,2d$.
Note that, by symmetry, $Lv$ then assumes the same value at all neighbors of 
$0$. We have
\[
Lv(\eta_j(0))=\frac{1}{\mu_0}\big(2d\beta-\alpha-(2d-1)\gamma\big),\quad 
j=1,\ldots,2d,
\]
and so the condition for $\gamma$ becomes
\begin{align} \label{gamma}
2d(\alpha-\beta)=2d\beta-\alpha-(2d-1)\gamma.
\end{align}
Selecting $\gamma=\gamma(\alpha,\beta,d)$ such that \eqref{gamma} is 
satisfied, we have by \eqref{psiform}, using the
same notation as above,
\begin{align*}
\Delta \Psi_{\Upsilon'}(v)(0) & =\frac{2d}{\mu_0^2} \, e^{\beta-\alpha} \Big( 
(2d-1)[e^{\gamma-2\beta+\alpha}-1]+
e^{2\alpha-2\beta}-1\Big)\\
& = \frac{2d}{\mu_0^2} \, e^{\beta-\alpha} 
\Big(e^{2\alpha-2\beta}-1-(2\alpha-2\beta)+(2d-1)
\big[e^{\gamma-2\beta+\alpha}-1-(\gamma-2\beta+\alpha)\big] \Big)\\
& = \frac{2d}{\mu_0^2} \, e^{\beta-\alpha}\big( \Upsilon(2\alpha-2\beta)+(2d-1) 
\Upsilon(\gamma-2\beta+\alpha)\big)\\
& = \frac{2d}{\mu_0^2} \, \exp\big({-\frac{\mu_0}{2d}Lv(0)}\big)\Big( 
\Upsilon\big( \frac{\mu_0}{d}Lv(0)   \big)+(2d-1) 
\Upsilon\big(-\frac{\mu_0}{d(2d-1)}Lv(0)\big)\Big)\\
& = F\big(Lv(0)\big),
\end{align*}
since
\[
(2d-1)(\gamma-2\beta+\alpha)=-2(\alpha-\beta).
\]
This proves the assertion as for any given $a>0$, we can clearly choose 
$\alpha$ and $\beta$ such that $Lv(0)=a$.
\end{proof}

\begin{example} \label{dgrid}
We consider the scaled $d$-dimensional integer lattice $(\tau \iZ)^d$ with 
scaling parameter $\tau>0$. So 
$V$ is the set of all points $(x_1,\ldots,x_d)$ where every $x_i$ is an 
integral multiple of $\tau$.
Let us assume that  all the weights are equal to $1$ and that $\mu(x)=\tau^2$ 
for all $x\in V$. That is, we have
\[
\Delta u(x)=\frac{1}{\tau^2}\, \sum_{i=1}^d \big(u(x+\tau e_i)-2u(x)+u(x-\tau 
e_i)\big),
\]
where $e_i$ denotes the $i$th unit vector. The graph is $2d$-regular and 
Ricci-flat. By \autoref{theo:CDRiccimain},
the condition CD($F_\tau$;0) holds with
\begin{align} \label{RicciFt}
F_\tau(a)=\frac{2d}{\tau^4} \exp\big(-\frac{\tau^2}{2d}a\big)\Big[ 
\Upsilon\Big(\frac{\tau^2}{d}a \Big)
+(2d-1) \Upsilon\Big(- \frac{\tau^2}{d(2d-1)}a \Big)\Big].
\end{align}
Since $\Upsilon(y)\sim \frac{1}{2}y^2$ as $y\to 0$, we obtain that as $a\to 0+$
\[
F_\tau(a)\sim \frac{2d}{\tau^4} \,\Big( \frac{1}{2} 
\,\frac{\tau^4}{d^2}\,a^2+\frac{2d-1}{2}\,
\frac{\tau^4}{d^2(2d-1)^2}\,a^2\Big) = \frac{2d}{d(2d-1)}\,a^2.
\]
In the same way we see that for fixed $a\ge 0$
\[
F_\tau(a)\to \frac{2d}{d(2d-1)}\,a^2\quad \mbox{as}\; \tau\to 0+.
\]
In particular, we obtain for $d=1$ that $F_\tau(a)$ tends as $\tau\to 0+$ to 
the 
quadratic function $2a^2$, which
appears in the classical (continuous) case in one dimension! 
\end{example}
 
\begin{theorem} \label{theo-CDRiccimainalpha}
Let $G=(V,E)$ be a $D$-regular unweighted Ricci-flat graph with $D\ge 2$. 
Assume 
that $\mu(y)=\mu_0>0$ for all $y\in V$. Then for any $\alpha\in 
(0,1)$, CD$_\alpha$($F_\alpha$;0) holds 
with
\begin{align} \label{RicciFalpha}
F_\alpha(a)=\frac{D}{\mu_0^2} \exp\big(-\frac{\mu_0(1-\alpha)}{D}a\big)\Big[ 
\exp\Big(\frac{2(1-\alpha)\mu_0}{D}a\Big)+\frac{1-\alpha}{\alpha}\exp\Big(-\frac
{2\alpha\mu_0}{D}a  \Big)-\frac{1}{\alpha}\Big].
\end{align}
\end{theorem}

\begin{proof} We first show that $F_\alpha$ is a strictly convex CD-function. 
Setting 
$\eta=\frac{\mu_0}{D}$ we can write
\begin{align*}
F_\alpha(a) & =\frac{1}{\alpha\eta\mu_0} e^{-\eta(1-\alpha)a}\big[ 
\alpha e^{2(1-\alpha)\eta a}+(1-\alpha)e^{-2\alpha\eta a}-1\big].
\end{align*}
Scaling the argument by putting $\tilde{a}=2\eta a$ and
introducing the function $\tilde{F}$ via $\tilde{F}(\tilde{a})=F_\alpha(a)$ we 
obtain
\[
\tilde{F}(\tilde{a})=\frac{1}{\alpha\eta\mu_0} e^{-\frac{1-\alpha}{2}a}\big[ 
\alpha e^{(1-\alpha)\tilde{a}}+(1-\alpha)e^{-\alpha\tilde{a}}-1\big].
\]
It now follows from \autoref{CDFexample} and \autoref{prop-perm1} that 
$\tilde{F}$, and thus also $F$, are
strictly convex CD-functions.

In what follows, we use the same notation as in the proof of 
\autoref{theo:CDRiccimain}.
Let $x\in V$ and suppose that $v\in \iR^V$ is such that 
\begin{align} \label{maxcondalpha}
{\mathcal L}_\alpha(v)(x)>0 \quad \mbox{and}\quad {\mathcal L}_\alpha(v)(x)\ge 
{\mathcal L}_\alpha(v)(\eta_j(x))\;\;\mbox{for 
all}\;j=1,\ldots, D.
\end{align}
Recall that we defined
\[
{\mathcal L}_\alpha(v)(x)  =-\frac{1}{\alpha}\,\Psi_{\Upsilon'}(\alpha v)(x).
\]
We have
\begin{align}
{\mathcal C}_\alpha(v)(x) & = \frac{1}{\mu_0}\,\sum_{i=1}^{D} e^{\alpha(z_i-z)}
\big[\Psi_{\Upsilon'}(v)(\eta_i(x))
-\Psi_{\Upsilon'}(v)(x)\big] \nonumber\\
& = \frac{1}{\mu_0^2}\,\sum_{i=1}^D e^{\alpha(z_i-z)}\sum_{j=1}^D \big[ 
\Upsilon'(z_{ij}-z_i)-\Upsilon'(z_j-z)\big] \nonumber\\
&  = \frac{1}{\mu_0^2}\,\sum_{i=1}^D e^{\alpha(z_i-z)}\sum_{j=1}^D \big[ 
e^{z_{ij}-z_i}-e^{z_j-z}\big] \nonumber\\
 & = \frac{1}{\mu_0^2}\,\sum_{j=1}^D e^{z_j-z} \sum_{i=1}^D 
\big(e^{z_{ij}-z_j-(1-\alpha)(z_i-z)}-e^{\alpha(z_i-z)}\big). 
\label{Calphaformula}
\end{align}

Let $i,j\in \{1,\ldots,D\}$. Then, by Young's inequality, we have for 
$a=e^{z_{ij}-z_j}>0$
and $b=e^{z_i-z}>0$
\[
a^\alpha=\left(\frac{a}{b^{1-\alpha}}\right)^\alpha b^{\alpha(1-\alpha)}\le 
\alpha\,\frac{a}{b^{1-\alpha}}
+(1-\alpha)b^\alpha,
\]
and thus
\[ 
\frac{a}{b^{1-\alpha}}\ge 
\frac{1}{\alpha}\,a^\alpha-\frac{1-\alpha}{\alpha}\,b^\alpha,
\]
which gives
\begin{align} 
e^{z_{ij}-z_j-(1-\alpha)(z_i-z)}-e^{\alpha(z_i-z)} & \ge 
\frac{1}{\alpha}\,e^{\alpha(z_{ij}-z_j)}
-\frac{1-\alpha}{\alpha}\,e^{\alpha(z_i-z)}-e^{\alpha(z_i-z)}\nonumber \\
 & =  
\frac{1}{\alpha}\,\Upsilon'\big(\alpha(z_{ij}-z_j)\big)-\frac{1}{\alpha}\,
\Upsilon'\big(\alpha(z_{i}-z)\big).
\label{youngalpha}
\end{align}
Using $z_{j^*j}=z$, inequality \eqref{youngalpha} for all $i\neq j^*$, 
\autoref{lem-Ricflat} (i) as well as the local maximum property 
\eqref{maxcondalpha}, we may now estimate as follows.
\begin{align*}
\sum_{i=1}^D 
& \,\big(e^{z_{ij}-z_j-(1-\alpha)(z_i-z)}-e^{\alpha(z_i-z)}\big)\ge 
e^{z-z_j-(1-\alpha)(z_{j^*}-z)}-e^{\alpha(z_{j^*}-z)}\\
& \;\;+\frac{1}{\alpha}\,
\sum_{i=1, i\neq j^*}^D 
\Big(\Upsilon'\big(\alpha(z_{ij}-z_j)\big)-\Upsilon'\big(\alpha(z_{i}
-z)\big)\Big)
\\
& \ge e^{z-z_j-(1-\alpha)(z_{j^*}-z)}-e^{\alpha(z_{j^*}-z)}+
\mu_0\big[{\mathcal L}_\alpha(v)(x)- 
{\mathcal L}_\alpha(v)(\eta_j(x))\big]\\
& 
\;\;-\frac{1}{\alpha}\,e^{\alpha(z-z_j)}+\frac{1}{\alpha}\,e^{\alpha(z_{j^*}-z)}
\\
& \ge 
e^{z-z_j-(1-\alpha)(z_{j^*}-z)}-\frac{1}{\alpha}\,e^{\alpha(z-z_j)}+\frac{
1-\alpha}{\alpha}\,
e^{\alpha(z_{j^*}-z)}.
\end{align*}  
Combining the last inequality and \eqref{Calphaformula} yields
\begin{align*}
{\mathcal C}_\alpha(v)(x) & \ge 
\frac{1}{\mu_0^2}\,\sum_{j=1}^D e^{z_j-z} \Big( 
e^{z-z_j-(1-\alpha)(z_{j^*}-z)}-\frac{1}{\alpha}\,e^{\alpha(z-z_j)}+\frac{
1-\alpha}{\alpha}\,
e^{\alpha(z_{j^*}-z)}\Big)\\
& = \frac{1}{\mu_0^2}\,\sum_{j=1}^D e^{(1-\alpha)(z_j-z)}
\Big( 
e^{(1-\alpha)(2z-z_j-z_{j^*})}+\frac{1-\alpha}{\alpha}\,e^{\alpha(z_j-2z+z_{j^*}
)}-\frac{1}{\alpha}\Big) \\
& = \frac{1}{\mu_0^2}\,\sum_{j=1}^D e^{(1-\alpha)(z_j-z)}
\Big( e^{2(1-\alpha)w_j}+\frac{1-\alpha}{\alpha}\,e^{-2\alpha 
w_j}-\frac{1}{\alpha}\Big).
\end{align*}
Assuming without restriction of generality that $z_1\le z_2\le \ldots \le 
z_D$, we can argue as in the proof of 
\autoref{theo:CDRiccimain} invoking the rearrangement inequality \eqref{Rea}. 
Thereby we obtain that 
\begin{align} \label{Calphaest2}
{\mathcal C}_\alpha(v)(x) \ge \frac{1}{\mu_0^2}\,\sum_{j=1}^D 
e^{(1-\alpha)(z_j-z)}
\Big( e^{2(1-\alpha)\tilde{w}_j}+\frac{1-\alpha}{\alpha}\,e^{-2\alpha 
\tilde{w}_j}-\frac{1}{\alpha}\Big).
\end{align}
Note that the term inside the brackets in \eqref{Calphaest2} is nonnegative. So 
we can symmetrize the
exponential factor in front of it and then use the convexity of the exponential 
function and $F_\alpha$
to get that
\begin{align*}
{\mathcal C}_\alpha(v)(x)  & \ge \frac{1}{2\mu_0^2}\,\sum_{j=1}^D 
\big[e^{(1-\alpha)(z_j-z)}+e^{(1-\alpha)(z_{j'}-z)}\big]
\Big( e^{2(1-\alpha)\tilde{w}_j}+\frac{1-\alpha}{\alpha}\,e^{-2\alpha 
\tilde{w}_j}-\frac{1}{\alpha}\Big)\\
& \ge \frac{1}{\mu_0^2}\,\sum_{j=1}^D 
e^{-(1-\alpha)\tilde{w}_j}
\Big( e^{2(1-\alpha)\tilde{w}_j}+\frac{1-\alpha}{\alpha}\,e^{-2\alpha 
\tilde{w}_j}-\frac{1}{\alpha}\Big)\\
& = \frac{1}{D}\,\sum_{j=1}^D F_\alpha\Big(\frac{D\tilde{w}_j}{\mu_0}\Big)
\ge F_\alpha\Big(\sum_{j=1}^D 
\frac{\tilde{w}_j}{\mu_0}\Big)=F_\alpha\big(Lv(x)\big).
\end{align*}
\end{proof}

\subsection{Examples of graphs that do not satisfy condition 
\texorpdfstring{CD($F$;0)}{CD(F;0)}}

In this section, we provide examples of graphs for which the condition 
CD($F$;0) does not hold. For the graphs under consideration, we construct a 
family of
functions $v$ such that ${\mathcal C}_0(v)=\Delta \Psi_{\Upsilon'}(v)$ becomes 
arbitrarily negative at a point $x_*$. Thus, there cannot be a CD-function 
$F$ with ${\mathcal C}_0(v)(x_*)\ge F\big(Lv(x_*)\big)$.

\begin{example} \label{exa:star}{\ }

\begin{minipage}{0.6\textwidth}

\medskip

We consider the unweighted graph $G=(V,E)$ with $V=\{x_*,x_1,x_2,x_3\}$, 
$E=\{x_*x_j:\,j=1,2,3\}$ with $\mu\equiv 1$ on $V$. Let $v\in \iR^V$ and set 
$z_j=v(x_j)$ for $j\in\{*,1,2,3\}$. 

At the vertex $x_*$,
we have

\end{minipage}
\begin{minipage}{0.4\textwidth}
\begin{figure}[H]
\begin{tikzpicture}[main_node/.style={circle,fill=black,draw,minimum 
size=0.2em,inner sep=1.5pt]}]

\node[main_node] [label={[distance=0.1cm]270:$x_3$}] (1) at (0,0) {};
\node[main_node] [label={[distance=0.2cm]90:$x_{\ast}$}](2) at (0, 1) {};
\node[main_node] [label={[distance=0.1cm]180:$x_1$}](3) at (-0.866, 1.5) {};
\node[main_node] [label={[distance=0.1cm]00:$x_2$}](4) at (0.866, 1.5) {};
 
\draw (1) -- (2) (3) -- (2) (4) -- (2);
\end{tikzpicture}
\caption{A star-like graph}
\label{fig:star}
\end{figure}
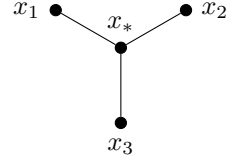
\end{minipage}

\medskip

\begin{align*}
\Delta \Psi_{\Upsilon'}(v)(x_*) = & 
\,\Psi_{\Upsilon'}(v)(x_1)+\Psi_{\Upsilon'}(v)(x_2)+\Psi_{\Upsilon'}(v)(x_{3})
-3\Psi_{\Upsilon'}(v)(x_*)\\
 = &\,\Upsilon'(z_*-z_1)+\Upsilon'(z_*-z_2)+\Upsilon'(z_*-z_3)\\
 &\;- 3\big[ \Upsilon'(z_1-z_*)+\Upsilon'(z_2-z_*)+\Upsilon'(z_3-z_*)\big]\\
 = & 
\,e^{z_*-z_1}+e^{z_*-z_2}+e^{z_*-z_3}-3-3\big[e^{z_1-z_*}+e^{z_2-z_*}+e^{z_3-z_*
}-3 \big]\\
 = & \,e^{a_1}+e^{a_2}+e^{a_3}-3-3\big[e^{-a_1}+e^{-a_2}+e^{-a_3}-3 \big],
\end{align*}
where we set $a_j=z_*-z_j$. We choose $v$ such that $z_*=0$, $-z_1=a_1=-t$ and 
$-z_j=a_j=t$ for $j=2,3$, where $t>0$
is a parameter. Then
\begin{align*}
Lv(x_*) & =a_1+a_2+a_3=t>0,\\
Lv(x_1) & = -a_1=t,\\
Lv(x_j) & = -a_j=-t,\;j=2,3.
\end{align*}
So we see that $Lv$ has a positive maximum at $x_*$. On the other hand, 
inserting the values of $a_j$, gives
\begin{align*}
\Delta \Psi_{\Upsilon'}(v)(x_*) & = e^{-t}+2e^{t}+6-3e^t-6e^{-t}=6-e^t-5e^{-t},
\end{align*}
which shows that 
\begin{align*}
\Delta \Psi_{\Upsilon'}(v)(x_*)\to -\infty \quad \mbox{as}\;t\to \infty.
\end{align*}    
\end{example}

\begin{example} \label{exa:hexa-cut}{\ }

\begin{minipage}{0.6\textwidth}

\medskip

We consider the graph from the previous example and add two edges at each 
of three
ends so that the resulting graph becomes a tree. More precisely, we have 
\begin{align*} 
V & =\{x_*,x_1,x_2,x_3,x_{11},x_{12},x_{21},x_{22},x_{31}, x_{32}\},\\
E & =\{x_*x_j:\,j=1,2,3\}\cup \{x_jx_{jk}:\,j=1,2,3,\,k=1,2\}.
\end{align*}
We consider the case without weights
and with $\mu\equiv 1$ on $V$. Let $v\in \iR^V$ and set $z_j=v(x_j)$ for 
$j\in\{*,1,2,3\}$ and
$z_{jk}=v(x_{jk})$ for $j\in\{1,2,3\}$ and $k\in\{1,2\}$.  As before, we put 
$a_j=z_*-z_j$ for $j=1,2,3$ and choose $v$ such that $z_*=0$, 
$-z_1=a_1=-t$, $-z_j=a_j=t$ for $j=2,3$, with $t>0$. As to the new vertices, we 
put $z_{jk}:=z_j$ for all $j=1,2,3$ and $k=1,2$. 
\end{minipage}
\begin{minipage}{0.4\textwidth}
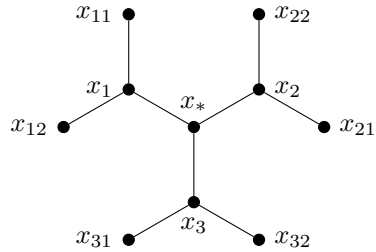
\begin{figure}[H]
\begin{tikzpicture}[main_node/.style={circle,fill=black,draw,minimum 
size=0.2em,inner sep=1.5pt]}]

    \node[main_node] [label={[distance=0.1cm]270:$x_3$}] (1) at (0,0) {};
    \node[main_node] [label={[distance=0.1cm]180:$x_{31}$}](2) at (-0.866, 
-0.5) {};
    \node[main_node] [label={[distance=0.1cm]00:$x_{32}$}](3) at (0.866, -0.5) 
{};
    \node[main_node] [label={[distance=0.2cm]90:$x_{\ast}$}](4) at (0, 1) {};
    \node[main_node] [label={[distance=0.1cm]180:$x_1$}](5) at (-0.866, 1.5) {};
    \node[main_node] [label={[distance=0.1cm]00:$x_2$}](6) at (0.866, 1.5) {};
    \node[main_node] [label={[distance=0.1cm]180:$x_{12}$}](7) at (-1.732, 1) 
{};
    \node[main_node] [label={[distance=0.1cm]180:$x_{11}$}](8) at (-0.866, 2.5) 
{};
    \node[main_node] [label={[distance=0.1cm]00:$x_{21}$}](9) at (1.732, 1) {};
    \node[main_node] [label={[distance=0.1cm]00:$x_{22}$}](10) at (0.866, 2.5) 
{};
    
    \draw (1) -- (2) (1) -- (3) (1) -- (4) -- (5) -- (7) (4) -- (6) (5) -- (8) 
(6) -- (10) (6) -- (9);
\end{tikzpicture}
\caption{Part of hexagonal tiling}
\label{fig:hexa-cut}
\end{figure}
\end{minipage}

\medskip

At the vertex $x_*$, we now obtain the same expression as above, since 
$\Upsilon'(0)=0$. Indeed,
\begin{align*}
\Delta \Psi_{\Upsilon'}(v)(x_*) = & 
\,\Psi_{\Upsilon'}(v)(x_1)+\Psi_{\Upsilon'}(v)(x_2)+\Psi_{\Upsilon'}(v)(x_{3})
-3\Psi_{\Upsilon'}(v)(x_*)\\
 = &\;\Upsilon'(z_*-z_1)+\Upsilon'(z_{11}-z_1)+\Upsilon'(z_{12}-z_1)\\
  &\,+\Upsilon'(z_*-z_2)+\Upsilon'(z_{21}-z_2)+\Upsilon'(z_{22}-z_2)\\
   &\,+\Upsilon'(z_*-z_3)+\Upsilon'(z_{31}-z_3)+\Upsilon'(z_{32}-z_3)\\
 &\;- 3\big[ \Upsilon'(z_1-z_*)+\Upsilon'(z_2-z_*)+\Upsilon'(z_3-z_*)\big]\\
 = &\,\Upsilon'(z_*-z_1)+\Upsilon'(z_*-z_2)+\Upsilon'(z_*-z_3)\\
 &\;- 3\big[ \Upsilon'(z_1-z_*)+\Upsilon'(z_2-z_*)+\Upsilon'(z_3-z_*)\big]\\
= & \;6-e^t-5e^{-t}.
\end{align*}
The values of $Lv$ on the set $\{x_*,x_1,x_2,x_3\}$ remain unchanged, since
\[
Lv(x_j)=3z_j-z_*-z_{j1}-z_{j2}=z_j-z_*=-a_j,\quad j=1,2,3.
\]
Thus $Lv$ has a positive local maximum at $x_*$ and as before $\Delta 
\Psi_{\Upsilon'}(v)(x_*)\to -\infty$
as the parameter $t\to \infty$.
\end{example}

\begin{example} \label{exa:hexagonal_tiling}{\ }

\begin{minipage}{0.4\textwidth}
Let us consider the graph, which is given by a hexagonal tiling of the plane. 
The graph shown in \autoref{fig:hexa-cut} 
obviously is a subgraph of this 
tiling. It follows from the previous example that there is no CD-function $F$ 
for which CD($F$;0) is satisfied.
\end{minipage}
\begin{minipage}{0.6\textwidth}
\begin{figure}[H]
\centering
\begin{tikzpicture}[x=4.34mm,y=7.5mm, rotate=90]
  \tikzset{
    box/.style={
      regular polygon,
      regular polygon sides=6,
      minimum size=10mm,
      inner sep=0mm,
      outer sep=0mm,
      rotate=0,
    draw
    }
  }

\foreach \i in {0,...,3} 
    \foreach \j in {0,...,3} {
            \node[box] at (2*\i,2*\j) {};
            \node[box] at (2*\i+1,2*\j+1) {};
            
     }

\end{tikzpicture}
\caption{Hexagonal tiling}
\label{fig:hexa-tiling}
\end{figure}
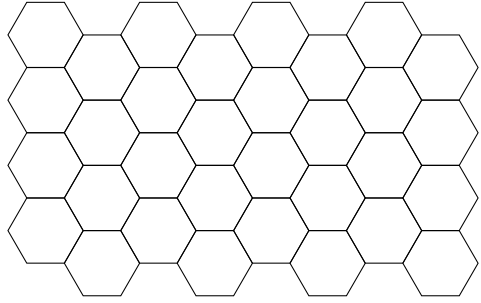
\end{minipage}

\end{example}

\section{Li-Yau inequalities on finite graphs}\label{sec:li-yau-finite}

\subsection{A new Li-Yau inequality} \label{sec-li-yau-null}

In this section we provide a proof of \autoref{theo:finite-graphs}. Let 
$G=(V,E)$ be a finite connected graph. Suppose that $u:[0,\infty)\times V 
\rightarrow (0,\infty)$ is a solution of the heat equation
on $G$, that is
\begin{align} \label{heat1}
\partial_t u-\Delta u=0\quad \mbox{in}\;[0,\infty)\times V.
\end{align}
Multiplying \eqref{heat1} by $u^{-1}$ and using identity \eqref{palmerel} we 
obtain for
$v:=\log u$ the equation
\begin{align} \label{heat2}
\partial_t v-\Delta v= \Psi_\Upsilon(v)\quad \mbox{in}\;(0,\infty)\times V.
\end{align}
Since 
\[
\Upsilon(z)+z=e^z-1=\Upsilon'(z),
\]
we can rewrite equation \eqref{heat2} as 
\begin{align} \label{heat3}
\partial_t v= \Psi_{\Upsilon'}(v)\quad \mbox{in}\;(0,\infty)\times V.
\end{align}

For the readers' convenience, let us repeat \autoref{theo:finite-graphs}.

\begin{theorem*} 
Assume that $G$ satisfies CD($F$;0) and let 
$\varphi$ be the relaxation function 
associated with the CD-function $F$. 
Suppose that $u:[0,\infty)\times V\rightarrow (0,\infty)$ is a solution of the 
heat equation
on $G$ (equation \eqref{heat1}). Then
\begin{align} \label{eq:LYfinite-copy}
-\Delta \log u\le \varphi(t)\quad \mbox{in}\;(0,\infty)\times V,
\end{align}
and thus
\begin{align} \label{eq:diffHarnack-copy}
\partial_t (\log u)\ge \Psi_\Upsilon(\log u)-\varphi(t)\quad 
\mbox{in}\;(0,\infty)\times V.
\end{align}
\end{theorem*}

\begin{proof} We define on $[0,\infty)\times V$ the function $G$ by setting
\[
G(t,x)=-\,\frac{1}{\varphi(t)}\,\Delta v(t,x)
=\,\frac{1}{\varphi(t)}\,L v(t,x),\quad t>0,\,x\in V,
\]
and $G(0,x)=0$, $x\in V$. Observe that $G$ is continuous in time, since 
$\varphi(t)\rightarrow \infty$ as $t \to 0+$.

Let $t_1>0$ be arbitrarily fixed. Suppose that $G$ (restricted to the set 
$[0,t_1]\times V$) assumes the global maximum at $(t_*,x_*)\in[0,t_1]\times V$ 
and that $G(t_*,x_*)>0$. By definition of $G$ it is clear that $t_*>0$, and thus
$(\partial_t G)(t_*,x_*)\ge 0$. 

Now, equation \eqref{heat3} implies that
\[
\partial_t \Delta v=\Delta \Psi_{\Upsilon'}(v)\quad \mbox{in}\;(0,\infty)\times 
V,
 \] 
which in turn entails that
\[
\partial_t G= - \varphi(t)^{-1} \Delta \Psi_{\Upsilon'}(v)-\dot{\varphi}(t) 
\varphi(t)^{-1}G.
\]
It follows that at the maximum point $(t_*,x_*)$ we have that
\[
 0 \le - \varphi(t)^{-1} \Delta \Psi_{\Upsilon'}(v)-\dot{\varphi}(t) 
\varphi(t)^{-1}G ,
\]
which is equivalent to
\begin{align} \label{est1}
 \Delta \Psi_{\Upsilon'}(v)\le -\dot{\varphi}(t) \varphi(t)^{-1} \,Lv.
\end{align}
at $(t_*,x_*)$.
Since $Lv(t_*,x_*)>0$ is the global maximum of $Lv(t_*,x)$ over $x\in V$, we 
may apply condition CD($F$;0), which gives
\begin{align} \label{est2}
F\big(Lv(t_*,x_*)\big)\le  \big(\Delta \Psi_{\Upsilon'}(v)\big)(t_*,x_*).
\end{align}

Setting $a=Lv(t_*,x_*) (>0)$, we infer from \eqref{est1} (at $(t_*,x_*)$) and 
\eqref{est2} that
\[
\frac{F(a)}{a}\le \frac{-\dot{\varphi}(t_*)}{\varphi(t_*)}.
\]
Since $\varphi$ satisfies the differential equation
\[
-\dot{\varphi}(t)= F(\varphi(t)),\quad t>0,
\]
it follows that
\[
\frac{F(a)}{a}\le \frac{-\dot{\varphi}(t_*)}{\varphi(t_*)}= 
\frac{F(\varphi(t_*))}{\varphi(t_*)},
\]
and thus, by strict monotonicity of $F(x)/x$,
\[
Lv(t_*,x_*)=a\le \varphi(t_*).
\]
This in turn gives
\[
G(t_*,x_*)\le 1.
\]
Since $(t_*,x_*)$ was a global maximum point of $G$ restricted to the set 
$[0,t_1]\times V$
with $t_1>0$ arbitrarily chosen, we obtain
\[
G(t_1,x)\le G(t_*,x_*)\le 1,\quad t_1\in (0,\infty),\,x\in V.
\]
This shows inequality \eqref{eq:LYfinite-copy}, which together with 
\eqref{heat2} 
implies the inequality \eqref{eq:diffHarnack-copy}.  
\end{proof}

\begin{example}
We consider the most simple case, i.e., the two point graph from 
\autoref{exa:twopoints}.
So $V=\{x_1,x_2\}$ and $\Delta u(x)=u(\tilde{x})-u(x)$, where $\tilde{x}$ 
denotes the only neighbor of $x\in V$.
Let $u:[0,\infty)\times V\to (0,\infty)$ be a solution of the heat equation on 
the graph. Then \autoref{theo:finite-graphs}
yields the estimate $-\Delta \log u\le \varphi(t)$ where $\varphi$ is given by
\[
\varphi(t)=\log\Big(\frac{1+e^{-2t}}{1-e^{-2t}}\Big)=-\log \big(\tanh 
t\big),\quad t>0,
\]
cf.\ \autoref{exa:twopoints}. Let us show that this estimate is optimal. 
Setting 
$u_i(t)=u(t,x_i)$ for $i=1,2$, the functions $u_1, u_2$ solve the ODE system
\[
\dot{u}_1=u_2-u_1,\quad
\dot{u}_2 =u_1-u_2,\quad t\ge 0.
\]
Adding the initial conditions $u_i(0)=u_i^0>0$, $i=1,2$, the solution is given 
by
\begin{align*}
u_1(t)&=\frac{1}{2}(u_1^0-u_2^0)e^{-2t}+\frac{1}{2}(u_1^0+u_2^0),\\
u_2(t)&=\frac{1}{2}(u_2^0-u_1^0)e^{-2t}+\frac{1}{2}(u_1^0+u_2^0).
\end{align*}
By symmetry, we may assume without loss of generality that $u_1^0\ge u_2^0$. 
This implies
$u_1(t)\ge u_2(t)$ for all $t\ge 0$, and thus
\[
-\Delta \log u(t,x_2)=\log\big(\frac{u_2(t)}{u_1(t)}\big)\le 0 \le 
\log\big(\frac{u_1(t)}{u_2(t)}\big)=-\Delta \log u(t,x_1)=:w(t).
\]
So we have to examine the function $w(t)$. Setting 
$\alpha=u_1^0/(u_1^0+u_2^0)>0$ and $\beta=u_2^0/(u_1^0+u_2^0)>0$ we have 
$\alpha+\beta=1$, and
for $t>0$ there holds
\[
w(t)=\log\Big(\frac{(\alpha-\beta)e^{-2t}+1}{(\beta-\alpha)e^{-2t}+1}\Big)
=\log\Big(\frac{(2\alpha-1)e^{-2t}+1}{(1-2\alpha)e^{-2t}+1}\Big)\le 
\log\Big(\frac{1+e^{-2t}}{1-e^{-2t}}\Big)=\varphi(t),
\]
here the upper estimate corresponds to the limiting and extreme case where 
$\alpha=1$ and $\beta=0$.
We can be arbitrarily close to this case by choosing $u_1^0$ and $u_2^0$ such 
that $u_1^0/u_2^0$
is sufficiently large, which shows that $-\Delta \log u\le \varphi(t)$ is a 
sharp estimate.

Note that as $t\to \infty$,
\begin{align*}
\varphi(t)=\log\Big(1+\frac{2e^{-2t}}{1-e^{-2t}} \Big)\sim 
\frac{2e^{-2t}}{1-e^{-2t}}\sim 2e^{-2t}.
\end{align*}
As $t\to 0$, we have
\begin{align*}
\varphi(t)=-\log\Big(\frac{1-e^{-2t}}{1+e^{-2t}}\Big)\sim -\log(t).
\end{align*}
\end{example}

\subsection{\texorpdfstring{$\alpha$-calculus and comparison with the square 
root approach}{alpha-calculus and comparison with the square 
root approach}} 
\label{sec-alpha}

Let $G=(V,E)$ be a finite connected graph. Suppose that $u:[0,\infty)\times V 
\rightarrow (0,\infty)$ is a solution of the heat equation on $G$. We have 
shown that $v=\log u$ satisfies the equation
\begin{align} \label{heat2p}
\partial_t v-\Delta v= \Psi_\Upsilon(v)\quad \mbox{in}\;(0,\infty)\times V \,.
\end{align}
Letting $\alpha\in (0,1)$, \eqref{heat2p} implies that
\[
\partial_t v-\Delta v-\frac{1}{\alpha}\,\Psi_{\Upsilon_\alpha} (v)= 
\Psi_\Upsilon(v)
-\frac{1}{\alpha}\,\Psi_{\Upsilon_\alpha} (v)\quad \mbox{in}\;(0,\infty)\times 
V.
\]
In view of the identity \eqref{palmalpha} we have
\[
\frac{\Delta( u^\alpha)}{\alpha u^\alpha}=\Delta 
v+\frac{1}{\alpha}\,\Psi_{\Upsilon_\alpha} (v)
=\frac{1}{\alpha}\,\Psi_{\Upsilon'}(\alpha v),
\]
and thus
\begin{align} \label{root1}
\partial_t v-\frac{\Delta( u^\alpha)}{\alpha u^\alpha}= \Psi_\Upsilon(v)
-\frac{1}{\alpha}\,\Psi_{\Upsilon_\alpha} (v)\quad \mbox{in}\;(0,\infty)\times 
V.
\end{align}
Note that in case $\alpha=\frac{1}{2}$, identity \eqref{Gammaroot} shows that 
equation \eqref{root1} then takes the form
\[
\partial_t v-2\,\frac{\Delta( \sqrt{u})}{\sqrt{u}}= 
2\,\frac{\Gamma(\sqrt{u})}{u}
\quad \mbox{in}\;(0,\infty)\times V.
\]

We are interested in an estimate of the form
\[
-\frac{\Delta( u^\alpha)}{\alpha 
u^\alpha}=-\frac{1}{\alpha}\,\Psi_{\Upsilon'}(\alpha v)
= {\mathcal L}_\alpha(v) \le \eta(t)
\quad \mbox{in}\;(0,\infty)\times V,
\]
for some appropriate positive function $\eta$. To achieve this, we first derive 
an equation for the
temporal derivative of the quantity to be estimated. Using the equation for $v$ 
we obtain
\begin{align*}
\partial_t \Big(\frac{1}{\alpha}\,\Psi_{\Upsilon'}(\alpha v)\Big)(t,x) & = 
 \frac{1}{\mu(x)}\sum_{y\sim x} w_{xy} e^{\alpha(v(t,y)-v(t,x))}\big(\partial_t 
v(t,y)-
 \partial_t v(t,x)\big)\\
 & =  \frac{1}{\mu(x)}\sum_{y\sim x} w_{xy} e^{\alpha(v(t,y)-v(t,x))}
 \big(\Psi_{\Upsilon'}(v)(t,y)-\Psi_{\Upsilon'}(v)(t,x)\big)\\
 & = {\mathcal C}_\alpha(v)(t,x),
\end{align*}
cf.\ \eqref{Calpha}. 

\begin{theorem} \label{Thmfinitealpha}
Let $\alpha\in(0,1)$, $G=(V,E)$ be a finite connected graph which satisfies 
CD$_\alpha$($F$;0) and let $\varphi$ be the relaxation function 
corresponding to the CD-function $F$. 
Suppose that $u:[0,\infty)\times V\rightarrow (0,\infty)$ is a solution of the 
heat equation
on $G$ (equation \eqref{heat1}). Then
\begin{align} \label{LYfinitealpha}
-\frac{\Delta( u^\alpha)}{\alpha 
u^\alpha}=-\frac{1}{\alpha}\,\Psi_{\Upsilon'}(\alpha \log u)
={\mathcal L}_\alpha(\log u)\le \varphi(t)\quad \mbox{in}\;(0,\infty)\times V,
\end{align}
and consequently
\begin{align} \label{diffHarnackalpha}
\partial_t (\log u)\ge  \Psi_\Upsilon(\log u)
-\frac{1}{\alpha}\,\Psi_{\Upsilon_\alpha} (\log u)-\varphi(t)\quad 
\mbox{in}\;(0,\infty)\times V.
\end{align}
\end{theorem}

\begin{proof} The proof is almost entirely analogous to the one of 
\autoref{theo:finite-graphs}, the only difference being that
in addition one uses the inequality
\[ 
F\big({\mathcal L}_\alpha(v)(t_*,x_*)\big)\le F\big(Lv(t_*,x_*)\big),
\]
which holds, since $F$ is strictly increasing and 
\[
{\mathcal 
L}_\alpha(v)(t_*,x_*)=Lv(t_*,x_*)-\frac{1}{\alpha}\,\Psi_{\Upsilon}(\alpha 
v)(t_*,x_*)\le Lv(t_*,x_*).
\]
\end{proof}

\section{Local Li-Yau inequalities and estimates on infinite 
graphs}\label{sec:li-yau-infinite}

The approach of \autoref{sec:li-yau-finite} to Li-Yau type estimates is 
restricted to finite graphs. Proofs of similar results in the case of infinite 
graphs are more involved. They additionally require the use of cut-off 
functions. The same 
difficulty arises when one
aims at local Li-Yau inequalities for positive functions that solve 
the heat equation only on a part
of the graph, e.g.\ in a ball. It turns out that the $\alpha$-calculus from 
\autoref{sec-alpha} is sufficiently robust to obtain the desired estimates.  
Throughout this section we confine ourselves to Ricci-flat graphs, which 
are introduced in \autoref{subsec:Ricci-flat}.

\subsection{General Ricci-flat graphs}

Throughout this subsection we assume that


\begin{align}\tag{R}\label{eq:R}
\begin{split}
&G=(V,E) \text{ is a } D\text{-regular unweighted Ricci-flat graph with } \\ 
&D\ge 2 \text{ and } \mu(y)=\mu_0>0 \text{ for all } y\in V \,.
\end{split}
\end{align}

We first need a slight generalization of the 
CD$_\alpha$($F_\alpha$; 0) 
provided by \autoref{theo-CDRiccimainalpha}.

\begin{corollary} \label{theo-CDRiccimainalphacutoff}
Assume \eqref{eq:R} and let $\alpha\in (0,1)$. Let $x_*\in V$ and $\psi\in 
\iR^V$ such 
that $\psi(x_*)>0$ and $\psi(y)>0$ for
all $y\sim x_*$. Let $v\in \iR^V$ and suppose that the function
\[
M(x):=\psi(x){\mathcal L}_\alpha(v)(x),\quad x\in V,
\] 
has a positive local maximum at $x_*$, that is,
\[
M(x_*)>0 \quad \mbox{and}\quad M(x_*)\ge M(y)
\;\;\mbox{for all}\;y\sim x.
\]
Then the following inequality holds true.
\begin{align} \label{CDalphacutoff}
{\mathcal C}_\alpha(v)(x_*)\ge 
F_\alpha\big(Lv(x_*)\big)-\frac{1}{\mu_0}{\mathcal L}_\alpha(v)(x_*) 
\sum_{y\sim x_*} 
e^{v(y)-v(x_*)} \frac{|\psi(x_*)-\psi(y)|}{\psi(y)}.
\end{align}
where $F_\alpha$ is the function given by \eqref{RicciFalpha}. 
\end{corollary}

\begin{proof} We follow the lines of the proof of 
\autoref{theo-CDRiccimainalpha} with $x$ replaced by $x_*$. When estimating the 
inner sum
we now have by the local maximum property of $M$
\begin{align*}
{\mathcal L}_\alpha(v)(x_*)- &\, {\mathcal L}_\alpha(v)(\eta_j(x_*))  =
\frac{1}{\psi(x_*)}\,\big(M(x_*)-M(\eta_j(x_*))\big)+\frac{
\psi(\eta_j(x_*))-\psi(x_*)}{\psi(\eta_j(x_*)) \psi(x_*)}\,
M(\eta_j(x_*))\\
& \ge -\frac{|\psi(\eta_j(x_*))-\psi(x_*)|}{\psi(\eta_j(x_*)) \psi(x_*)}\,
M(x_*)= -{\mathcal 
L}_\alpha(v)(x_*)\,\frac{|\psi(x_*)-\psi(\eta_j(x_*))|}{\psi(\eta_j(x_*))}.
\end{align*} 
Arguing as in the proof of \autoref{theo-CDRiccimainalpha}, the last term leads 
to the second term on the right of  
inequality \eqref{CDalphacutoff}. 
\end{proof}

\begin{lemma} \label{trivialalpha}
Assume \eqref{eq:R}. Let $\alpha\in (0,1)$, $x_*\in V$ and $u:\,V\to 
(0,\infty)$ such 
that $\Delta (u^\alpha)(x_*)<0$.
Then one has
\begin{align}
\sum_{y\sim x_*} u(y)\le D^{1/\alpha} u(x_*).
\end{align} 
\end{lemma}

\begin{proof}
For positive numbers $a_j$, $j=1,\ldots,D$ we have the inequality
\[
\big(\sum_{j=1}^D a_j\big)^\alpha \le  \sum_{j=1}^D a_j^\alpha, 
\]
and thus
\begin{align*}
\sum_{y\sim x_*} \frac{u(y)}{u(x_*)} & \le  \big(\sum_{y\sim x_*} 
\frac{u(y)^\alpha }{u(x_*)^\alpha}\big)^{1/\alpha}\\
& =\Big( \frac{1}{u(x_*)^\alpha}\,\sum_{y\sim x_*} 
\big(u(y)^\alpha-u(x_*)^\alpha\big) +D  \Big)^{1/\alpha}
\le D^{1/\alpha}.
\end{align*}
\end{proof}

The following lemma is a very useful auxiliary result when proving estimates 
involving CD-functions.

\begin{lemma} \label{lem-superadd}
Let $g:[0,\infty)\rightarrow [0,\infty)$ be a strictly increasing $C^1$ 
function with $g(0)=0$.
Assume that there exists $a_*>0$ such that $c_0:=\min_{[0,a_*]} g'>0$ and 
$g|_{[a_*,\infty)}$
is convex. Let $c_1=\max_{[0,a_*]} g'$ and set $\gamma=c_1/c_0$. Then there 
holds
\begin{align} \label{superH}
g(x)+g(y)\le g(x+\gamma y),\quad x,y\in [0,\infty).
\end{align}
\end{lemma} 

\begin{proof}
Let $x,y\in [0,\infty)$. It suffices to show \eqref{superH} in the 
case $x\ge y$. In fact,
if $x<y$, we have (since $\gamma\ge 1$)
\[
y+\gamma x\le x+\gamma y,
\]
and thus $g(y+\gamma x)\le g(x+\gamma y)$.

So let us assume that $x\ge y$. We consider two cases. Suppose first that $y\le 
a_*$.
Set $\xi=x+\gamma y$. Then
\begin{align*}
g(x)+g(y)\le g(x)+c_1 y=g(x)+c_0(\xi-x)\le g(\xi),
\end{align*}
since $g'\ge c_0$ in $[0,\infty)$. So the desired inequality holds.

Now suppose that $a_*<y\,(\le x)$. By convexity of $g$ in $[a_*,\infty)$ and 
since $g(a_*)\le c_1 a_*$,
we have
\begin{align*}
g(y) & \le g(a_*)+g'(y)(y-a_*) \le c_1 a_*+g'(y)(y-a_*)\\
& \le \gamma g'(y)y+c_1 a_*-g'(y)\big((\gamma-1)y+a_*\big)\\
& \le \gamma g'(y)y+c_1 a_*-c_0\big((\gamma-1)y+a_*\big)\\
& = \gamma g'(y)y-(c_1-c_0)(y-a_*)\le \gamma g'(y)y.
\end{align*}
Using this inequality and the convexity of $g$ in $[a_*,\infty)$, it follows 
that
\begin{align*}
g(x)+g(y) & \le g(x)+\gamma g'(y)y\le g(x)+\gamma g'(x)y\\
& = g(x)+\big((x+\gamma y)-x \big)g'(x)\le g(x+\gamma y).
\end{align*}
This proves the lemma.
\end{proof}

The main result in this subsection is the following.

\begin{theorem} \label{thm-alphainfinity}
Assume \eqref{eq:R}. Let $x_0\in V$ and $r\in \iN$. Let the function 
$u:\,[0,\infty)\times V\to (0,\infty)$ be 
a solution of the heat equation on the ball $\bar{B}_{2r}(x_0)=\{y\in 
V:\,d(y,x_0)\le 2r\}$, that is,
\[
\partial_t u(t,x)-\Delta u(t,x)=0,\quad t\ge 0,\,x\in\bar{B}_{2r}(x_0) .
\]
Then for any $\alpha\in (0,1)$ there exists a constant $C=C(\alpha,\mu_0,D)>0$ 
such that
\begin{align}
\partial_t (\log u)\ge \Psi_\Upsilon(\log u)
-\frac{1}{\alpha}\,\Psi_{\Upsilon_\alpha} (\log 
u)-\varphi_\alpha(t)-\frac{C}{r}\quad \mbox{in}\;(0,\infty)\times 
\bar{B}_{r}(x_0),
\end{align}
where $\varphi_\alpha$ is the relaxation function corresponding to the 
CD-function $F_\alpha$
given in \eqref{RicciFalpha}.
\end{theorem}

\begin{proof} Setting $v=\log u$ we know from \autoref{sec-alpha} that
\begin{align} \label{infalpha0}
\partial_t v+{\mathcal L}_\alpha(v)=\Psi_\Upsilon(v)
-\frac{1}{\alpha}\,\Psi_{\Upsilon_\alpha} (v)\quad \mbox{in}\;(0,\infty)\times 
\bar{B}_{2r}(x_0),
\end{align}
and
\begin{align} \label{infalpha1}
\partial_t {\mathcal L}_\alpha(v)= - {\mathcal C}_\alpha(v)\quad 
\mbox{in}\;(0,\infty)\times 
\bar{B}_{2r}(x_0).
\end{align}
We define a cut-off function $\psi:\,V\to [0,\infty)$ by
\begin{align} \label{psidef}
\psi(x)=\left\{ \begin{array}{l@{\;:\;}l}
0 & 2r<d(x,x_0) \\
\frac{2r-d(x,x_0)}{r} & r\le d(x,x_0)\le 2r\\
1 & d(x,x_0)<r.
\end{array} \right.
\end{align}
Let $\alpha\in (0,1)$ be fixed and set $\varphi(t)=\varphi_\alpha(t)$. We 
consider the quantities $G$ and $\tilde{G}$ defined on $[0,\infty)\times V$ by
\[
G(t,x)=\frac{{\mathcal L}_\alpha(v)(t,x)}{\varphi(t)}\quad \mbox{and}\;\;
\tilde{G}(t,x)=\psi(x)G(t,x),\quad t>0,\,x\in V,
\]
and $G(0,x)=\tilde{G}(0,x)=0$, $x\in V$. Note that $G$ and $\tilde{G}$ are 
continuous in time, since $\varphi(t)\rightarrow \infty$ as $t \to 0+$.

Multiplying \eqref{infalpha1} by $\psi(x)\varphi(t)^{-1}$ we obtain
\begin{align} \label{infalpha2}
\partial_t \tilde{G}(t,x)= -\varphi(t)^{-1}\psi(x) {\mathcal C}_\alpha(v)(t,x)
-\dot{\varphi}(t) \varphi(t)^{-1}\tilde{G}(t,x)\quad 
\mbox{in}\;(0,\infty)\times 
\bar{B}_{2r}(x_0).
\end{align}

Let $t_1>0$ be arbitrarily fixed. Suppose that $\tilde{G}$ (restricted to the 
set 
$[0,t_1]\times \bar{B}_{2r}(x_0)$) assumes the global maximum at 
$(t_*,x_*)\in[0,t_1]\times 
\bar{B}_{2r}(x_0)$ and that 
$\tilde{G}(t_*,x_*)>0$. By definition of $\tilde{G}$ it is clear that $t_*>0$ 
and $x_*\in 
B_{2r}(x_0)=\{y\in V:\,d(y,x_0)< 2r\}$. In particular, we have
$(\partial_t \tilde{G})(t_*,x_*)\ge 0$. 

We now distinguish three cases.

{\bf Case 1:} $x_*\in B_r(x_0)$. Then $\psi(x_*)=1$ and also 
$\psi(y)=1$ for all $y\sim x$. Thus
${G}(t_*,x_*)=\tilde{G}(t_*,x_*)\ge \tilde{G}(t,y)=G(t,y)$ for all $t\in 
[0,t_1]$ and all $y\sim x$.
In particular, ${\mathcal L}_\alpha(v)(t_*,x)\ge {\mathcal L}_\alpha(v)(t_*,y)$ 
for all $y\sim x$. So we can apply CD$_\alpha$($F_\alpha$; 0) from 
\autoref{theo-CDRiccimainalpha}
(with local maximum of ${\mathcal L}_\alpha(v)$), thereby obtaining 
at the maximum point that
\[
\psi(x_*) \varphi^{-1}(t_*)F_\alpha\big({\mathcal L}_\alpha(v)(t_*,x)\big)\le
  -\psi(x_*)\dot{\varphi}(t_*) \varphi^{-1}(t_*)G(t_*,x_*).
\]
We can then argue as in the proof of \autoref{Thmfinitealpha} (see also 
\autoref{theo:finite-graphs}) to find that $G(t_*,x_*)\le 1$, which implies
\[
\tilde{G}(t_1,x)\le \tilde{G}(t_*,x_*)=G(t_*,x_*)\le 1,\quad x\in 
\bar{B}_{2r}(x_0).
\]

{\bf Case 2:} $d(x_*,x_0)=2r-1$. In this case $\psi(x_*)=\frac{1}{r}$. Here we 
estimate $\tilde{G}$ directly without using \eqref{infalpha2}. We have for 
arbitrary $t>0$ and $x\in V$
\begin{align*}
{\mathcal L}_\alpha(v)(t,x)  =-\frac{\Delta( u^\alpha)(t,x)}{\alpha 
u(t,x)^\alpha}
=\frac{1}{\alpha \mu_0}\,\sum_{y\sim 
x}\Big(1-\frac{u(t,y)^\alpha}{u(t,x)^\alpha}\Big)
\le \frac{D}{\alpha \mu_0},
\end{align*}
by positivity of $u$. Hence
\[
\tilde{G}(t_1,x)\le \tilde{G}(t_*,x_*)=\frac{\psi(x_*){\mathcal 
L}_\alpha(v)(t,x)}{\varphi(t)}
\le \frac{D}{\alpha\mu_0 r \varphi(t_*)}\le \frac{D}{\alpha\mu_0 r 
\varphi(t_1)},\quad x\in 
\bar{B}_{2r}(x_0),
\]
since $\varphi$ is non-increasing.

{\bf Case 3:} $\psi(x_*)=s/r$ with $s\in \{2,\ldots,r\}$. Here $\psi(y)>0$ for 
all $y\sim x$, and 
thus we may apply \autoref{theo-CDRiccimainalphacutoff} at the time level 
$t_*$. 
From \eqref{infalpha2} we then obtain at the maximum point $(t_*,x_*)$
\begin{align} \label{infalpha3}
F_\alpha\big(Lv(t_*,x_*)\big)\le  {\mathcal 
L}_\alpha(v)(t_*,x_*)\Big(\frac{1}{\mu_0}\, \sum_{y\sim x_*} 
e^{v(t_*,y)-v(t_*,x_*)} \frac{|\psi(x_*)-\psi(y)|}{\psi(y)}-\dot{\varphi}(t_*) 
\varphi(t_*)^{-1}\Big).
\end{align}
By definition of $\psi$, we have for all $y\sim x_*$
\[
 \frac{|\psi(x_*)-\psi(y)|}{\psi(y)}\le \frac{1/r}{(s-1)/r}\,=\frac{1}{s-1}.
\]
Further, we know that $\frac{\Delta( u^\alpha)(t_*,x_*)}{\alpha 
u(t_*,x_*)^\alpha}<0$, which
implies $\Delta( u^\alpha)(t_*,x_*)<0$ as well, by positivity of $u$. 
\autoref{trivialalpha} then yields
\[
\sum_{y\sim x_*} u(t_*,y)\le D^{1/\alpha} u(t_*,x_*),
\]
and thus
\begin{align*}
\sum_{y\sim x_*} 
e^{v(t_*,y) \,-v(t_*,x_*)} & \frac{|\psi(x_*)-\psi(y)|}{\psi(y)} \le 
\frac{1}{s-1}\,\sum_{y\sim x_*} 
e^{v(t_*,y)-v(t_*,x_*)}\\
& = \frac{1}{s-1}\,\sum_{y\sim x_*} \frac{u(t_*,y)}{u(t_*,x_*)}\le  
\frac{D^{1/\alpha}}{s-1}.
\end{align*}
Using this estimate, together with the ODE for the relaxation function 
$\varphi$ and
\[
F_\alpha\big( {\mathcal L}_\alpha(v)(t_*,x_*)\big)\le 
F_\alpha\big(Lv(t_*,x_*)\big),
\]
it follows from \eqref{infalpha3} that
\[
F_\alpha\big( {\mathcal L}_\alpha(v)(t_*,x_*)\big)\le  
{\mathcal L}_\alpha(v)(t_*,x_*)\Big(\frac{D^{1/\alpha}}{\mu_0 
(s-1)}\,+\frac{F_\alpha\big(\varphi(t_*)
\big)}{ \varphi(t_*)}\Big).
\]
We define the function $H:\,[0,\infty)\to [0,\infty)$ by $H(0)=0$ and 
$H(x)=F_\alpha(x)/x$ for $x>0$.
We also put $\omega_s= \frac{D^{1/\alpha}}{\mu_0 (s-1)}$. Suppressing the 
arguments, the last
inequality  is then equivalent to
\begin{align} \label{infalpha4}
H\big( {\mathcal L}_\alpha(v)\big)\le H(\varphi)+\omega_s.
\end{align}
By \autoref{CDFexample}, we may apply \autoref{lem-superadd} to the function 
$g=H$. 
Let $\gamma=\gamma(\alpha,D,\mu_0)>0$
be the corresponding constant. Then \autoref{lem-superadd} gives
\[
H(\varphi)+\omega_s= H(\varphi)+H\big( H^{-1}(\omega_s)\big)\le 
H\big(\varphi+\gamma H^{-1}(\omega_s)\big),
\]
which when combined with \eqref{infalpha4} yields
\[
H\big( {\mathcal L}_\alpha(v)\big)\le H\big(\varphi+\gamma 
H^{-1}(\omega_s)\big).
\]
Since $H$ is strictly increasing, we deduce that
\[
{\mathcal L}_\alpha(v)(t_*,x_*)\le \varphi(t_*)+\gamma H^{-1}(\omega_s),
\]
that is
\[
G(t_*,x_*)\le 1+\varphi(t_*)^{-1}\gamma H^{-1}(\omega_s).
\]
We now find that
\begin{align*}
\tilde{G}(t_1,x) & \le \tilde{G}(t_*,x_*)=\frac{s}{r}G(t_*,x_*)\\
& \le \frac{s}{r}\big(1+\varphi(t_*)^{-1}\gamma H^{-1}(\omega_s)\big)\\
& \le 1+\frac{\gamma C_s}{r} \varphi(t_1)^{-1},
\end{align*}
where $C_s:=s H^{-1}(\omega_s)$. It is now not difficult to check that there 
exists a number 
$M=M(\alpha,D,\mu_0)>0$ such that $C_s\le M$ for all $s\ge 2$ (recall that 
$H(x)$ behaves as a linear function as $x\to 0$).
It follows that
\[
\tilde{G}(t_1,x)\le  1+\frac{\gamma M}{r} \varphi(t_1)^{-1},\quad x\in 
\bar{B}_{2r}(x_0).
\]

Combining all three cases we see that for arbitrary $t_1>0$
\[
\tilde{G}(t_1,x)\le  1+\max\Big\{\gamma M,\frac{D}{\alpha\mu_0}\Big\}\,\frac{ 
\varphi(t_1)^{-1}}{r},\quad x\in 
\bar{B}_{2r}(x_0),
\]
which implies that
\[
{\mathcal L}_\alpha(v)(t,x)\le  \varphi(t)+\max\Big\{\gamma 
M,\frac{D}{\alpha\mu_0}\Big\}\,\frac{1}{r},\quad t>0,\,x\in 
\bar{B}_{r}(x_0).
\]
This together with \eqref{infalpha0} proves the theorem.
\end{proof}

We remark that 
\[
\varphi_\alpha(t)\sim \frac{D}{2(1-\alpha)t}\quad \mbox{as}\;t\to \infty.
\]
This follows from \autoref{lem:asymptrelax} and the fact that $F_\alpha(a)\sim 
\frac{2(1-\alpha)}{D}a^2$ as $a\to 0+$.
Using \autoref{lem:asymptrelax} we also see that
\[
\varphi_\alpha(t)\sim -\frac{D}{\mu_0(1-\alpha)}\,\log t\quad \mbox{as}\;t\to 
0+.
\]

\subsection{The example \texorpdfstring{$\iZ$}{Z}}\label{sec:example-Z}
In the special case where the graph is given by the lattice $\iZ$ we can 
improve the 
estimate of \autoref{thm-alphainfinity}
in two ways. On the one hand, we are able to treat the limit case 
$\alpha=0$. On the other hand, we obtain
an estimate with the relaxation function $\varphi$ associated to the full
CD-function $F$ given by
\eqref{RicciF} with $D=2$ and $\mu_0=1$. This is possible due to the special 
structure of the term
$\Delta \Psi_{\Upsilon'}(v)$. Note that
\[ 
\varphi(t)\sim \frac{1}{2t} \quad \mbox{as}\;t\to \infty,
\]
and
\[
\varphi(t)\sim -2\log t \quad \mbox{as}\;t\to 0+,
\]
by \autoref{lem:asymptrelax}.

\begin{lemma} \label{ThetaInequ}
Let $G=(V,E)$ be the lattice $\iZ$ without weights and with $\mu\equiv 1$ on 
$V=\iZ$.
Let $\eta_j:\,\iZ\to \iZ$, $j=1,2$ be defined by $\eta_1(x)=x-1$ and 
$\eta_2(x)=x+1$.
Then for any $v\in \iR^V$ and $x\in \iZ$ there holds
\[
\Delta \Psi_{\Upsilon'}(v)(x)\ge 
2e^{-\frac{Lv(x)}{2}}\Upsilon(Lv(x))+\Theta(v)(x)
\]
where
\begin{align} \label{ThetaDef}
\Theta(v)(x) =\sum_{j=1}^2 e^{v(\eta_j(x))-v(x)}
\Big(e^{-Lv(\eta_j(x))}-1+Lv(x)\Big).
\end{align}
Moreover, if $Lv(x)\ge 1$, we have
\[
\Theta(v)(x) \ge 2e^{-\frac{Lv(x)}{2}}\big(Lv(x)-1\big).
\]
\end{lemma}

\begin{proof} 
We use the same notation as in the proof of \autoref{theo:CDRiccimain}. 
Following the first lines
from there and observing that
\[
2w_j=2z-z_j-z_{j^*}=Lv(x),\quad j=1,2,
\]
we see that
\begin{align*}
\Delta \Psi_{\Upsilon'}(v)(x) & = 
\sum_{j=1}^2 e^{z_j-z} \big(e^{2w_j}-1+e^{z_{jj}-2z_j+z}-1\big)\\
& = \sum_{j=1}^2 e^{z_j-z} 
\big(\Upsilon(Lv(x))+e^{-Lv(\eta_j(x))}-1+Lv(x)\big)\\
& = \Upsilon(Lv(x))  \sum_{j=1}^2 e^{z_j-z}+\Theta(v)(x)\\
& \ge 2e^{-\frac{Lv(x)}{2}}\Upsilon(Lv(x))+\Theta(v)(x),
\end{align*}
by convexity of the exponential function. The last assertion follows from the 
definition of $\Theta(v)$
and the same convexity inequality we used before.
\end{proof}

The following simple fact will be needed in our argument.

\begin{lemma} \label{lem-eta}
Let $\eta> 1$ and $f:[0,\infty)\rightarrow \iR$ be given by
\[
f(a)=e^{-\eta a}-1+a.
\]
Then $f$ has exactly two zeros: $a=0$ and $a=a_*(\eta)>0$, where
\begin{align} \label{asym}
a_*(\eta)\sim 2(\eta-1)\quad \mbox{as}\;\eta\to 1.
\end{align}
\end{lemma}

\begin{proof}
It is clear that $a_*(\eta)\to 0$ as $\eta\to 1$. Thus, as $\eta\to 1$ we have 
by
expanding the exponential function around $0$,
\begin{align*}
0 & =e^{-\eta a_*(\eta)}-1+a_*(\eta) = \big(1-\eta a_*(\eta)+\frac{1}{2} \eta^2 
a_*(\eta)^2\big) 
-1+a_*(\eta)+ O([-\eta a_*(\eta)]^3)\\
& = a_*(\eta) \Big( \frac{1}{2} \eta^2 a_*(\eta)-(\eta-1)\Big) 
+O([-a_*(\eta)]^3).
\end{align*}
This implies that the term inside the big brackets tends to $0$ as $\eta\to 1$, 
which in turn
entails \eqref{asym}. 
\end{proof}

The main result of this subsection reads as follows.

\begin{theorem} \label{thm-Z}
Let $G=(V,E)$ be the lattice $\iZ$ without weights and with $\mu\equiv 1$ on 
$V=\iZ$.
Let $x_0\in \iZ$ and $r\in \iN$. Let the function $u:\,[0,\infty)\times \iZ\to 
(0,\infty)$ be 
a solution of the heat equation on the ball $\bar{B}_{2r}(x_0)$.
Then there exists a constant $C>0$ such that
\begin{align}
\partial_t (\log u)\ge \Psi_\Upsilon(\log u)
-\varphi(t)-\frac{C}{r}\quad \mbox{in}\;(0,\infty)\times 
\bar{B}_{r}(x_0),
\end{align}
where $\varphi$ is the relaxation function corresponding to the CD-function $F$
given in \eqref{RicciF} with $D=2$ and $\mu_0=1$.
\end{theorem}

\begin{proof} Setting $v=\log u$ we know from \autoref{sec-li-yau-null} that
\begin{align} \label{infZ0}
\partial_t v+Lv=\Psi_\Upsilon(v)
\quad \mbox{in}\;(0,\infty)\times 
\bar{B}_{2r}(x_0),
\end{align}
and
\begin{align} \label{infZ1}
\partial_t Lv= -\Delta \Psi_{\Upsilon'}(v)=-{\mathcal C}_0(v)\quad 
\mbox{in}\;(0,\infty)\times 
\bar{B}_{2r}(x_0).
\end{align}
Let $\psi$ be the cut-off function from \eqref{psidef} and define the functions 
$G$ and $\tilde{G}$ on $[0,\infty)\times \iZ$ by
\[
G(t,x)=\frac{Lv(t,x)}{\varphi(t)}\quad \mbox{and}\;\;
\tilde{G}(t,x)=\psi(x)G(t,x),\quad t>0,\,x\in \iZ,
\]
and $G(0,x)=\tilde{G}(0,x)=0$, $x\in \iZ$. Multiplying \eqref{infZ1} by 
$\psi(x)\varphi(t)^{-1}$ we obtain
\begin{align} \label{infZ2}
\partial_t \tilde{G}= -\varphi^{-1}\psi {\mathcal C}_0(v)
-\dot{\varphi} \varphi^{-1}\tilde{G}\quad \mbox{in}\;(0,\infty)\times 
\bar{B}_{2r}(x_0).
\end{align}

Let $t_1>0$ be arbitrarily fixed. Suppose that $\tilde{G}$ (restricted to the 
set 
$[0,t_1]\times \bar{B}_{2r}(x_0)$) attains the global maximum at 
$(t_*,x_*)\in[0,t_1]\times 
\bar{B}_{2r}(x_0)$ and that 
$\tilde{G}(t_*,x_*)>0$. Then $t_*>0$, $x_*\in 
B_{2r}(x_0)$, and in particular we have
$(\partial_t \tilde{G})(t_*,x_*)\ge 0$. 

We now distinguish three cases.

{\bf Case 1:} $x_*\in {B}_{r}(x_0)$. Then $\psi(x_*)=1$ and
$\psi(y)=1$ for all $y\sim x$. Consequently,
${G}(t_*,x_*)=\tilde{G}(t_*,x_*)\ge \tilde{G}(t,y)=G(t,y)$ for all $t\in 
[0,t_1]$ and all $y\sim x$.
In particular, $Lv(t_*,x)\ge Lv(t_*,y)$ for all $y\sim x$. By the CD-inequality 
from 
\autoref{theo:CDRiccimain} we obtain
at the maximum point that
\[
\psi \varphi^{-1}F(Lv)\le  -\psi\dot{\varphi} \varphi^{-1}G.
\]
We can then argue as in the proof of \autoref{theo:CDRiccimain}, thereby 
getting that $G(t_*,x_*)\le 1$, which implies that
\[
\tilde{G}(t_1,x)\le \tilde{G}(t_*,x_*)=G(t_*,x_*)\le 1,\quad x\in 
\bar{B}_{2r}(x_0).
\]

{\bf Case 2:} $d(x_0,x_*)=2r-1$, that is, $\psi(x_*)=\frac{1}{R}$. From 
\autoref{ThetaInequ} we know that
\begin{align*}
{\mathcal C}_0(v) \ge 2e^{-\frac{Lv}{2}}\Upsilon(Lv)+\Theta(v)
\end{align*}
where $\Theta(v)$ is given by \eqref{ThetaDef}. Note that $\Theta(v)\ge 0$ if 
$Lv(x)\ge 1$. If this is the case we even have an 
estimate of the form
\begin{align} \label{Thetabound}
\Theta(v)\ge 2e^{-\frac{Lv}{2}}\big(Lv-1\big).
\end{align}

{\em Case 2a:} Suppose that $Lv(t_*,x_*)\le 1$. Then
\begin{align*}
\tilde{G}(t_1,x) & \le \tilde{G}(t_*,x_*) =\psi(x_*)\varphi(t_*)^{-1} 
Lv(t_*,x_*)\\
& \le \frac{1}{r}\varphi(t_*)^{-1} \le  \frac{1}{r}\varphi(t_1)^{-1},\quad x\in 
\bar{B}_{2r}(x_0),
\end{align*}
since $\varphi$ is decreasing.

{\em Case 2b:} Suppose now that $Lv(t_*,x_*)> 1$. Now we use \eqref{infZ2}. At 
the maximum point, we can bound $\Theta(v)$ from below by the bound given in 
\eqref{Thetabound},
and so we obtain
\[ 
\psi \varphi^{-1}\cdot 2e^{-\frac{Lv}{2}}\Big(\Upsilon(Lv)+Lv-1\Big)\le  
-\psi\dot{\varphi} \varphi^{-1}G. 
\]
This implies at the maximum point
\[
2e^{-\frac{Lv}{2}}\Big(\Upsilon(Lv)+Lv-1\Big)\le  -\dot{\varphi} \varphi^{-1}Lv,
\]
which is equivalent to
\[
F(Lv)=2e^{-\frac{Lv}{2}}\Big(\Upsilon(Lv)+\Upsilon(-Lv)\Big)\le  -\dot{\varphi} 
\varphi^{-1}Lv
+2e^{-\frac{Lv}{2}} e^{-Lv}.
\]
Since $Lv(t_*,x_*)> 1$, the last inequality implies
\[
F(Lv)\le  -\dot{\varphi} \varphi^{-1}Lv
+2e^{-\frac{3}{2}}Lv.
\] 
Setting $H(0)=0$ and $H(x)=F(x)/x$, $x>0$, and $\omega=2/\sqrt{e^3}$ this can 
be 
rewritten as
\[
H(Lv)\le  -\dot{\varphi} \varphi^{-1}
+\omega.
\]
Since $-\dot{\varphi}=F(\varphi)$, we thus get (still at the maximum point)
\begin{align*}
H(Lv) & \le H(\varphi)+\omega=H(\varphi)+H\big( H^{-1}(\omega)\big)\\
& \le H\big(\varphi+\gamma H^{-1}(\omega)\big),
 \end{align*} 
where we used \autoref{lem-superadd} with corresponding constant $\gamma>0$; 
this lemma
applies to $H$ thanks to \autoref{CDFexample}. Since $H$ is strictly 
increasing, 
the last inequality implies 
that
\[
Lv(t_*,x_*)\le \varphi(t_*)+\gamma H^{-1}(\omega),
\]
that is,
\[
G(t_*,x_*)\le 1+\varphi(t_*)^{-1} \gamma H^{-1}(\omega). 
\]
We now obtain
\begin{align*}
\tilde{G}(t_1,x) & \le \tilde{G}(t_*,x_*)=\frac{1}{r}\,G(t_*,x_*)\\
& \le \frac{1}{r}\,\big(1+\varphi(t_*)^{-1} \gamma H^{-1}(\omega)\big)\\
& \le \frac{1}{r}\,\big(1+\varphi(t_1)^{-1} \gamma H^{-1}(2/\sqrt{e^3})\big),
\end{align*}
where in the last step we used the fact that $\varphi$ is decreasing. 

{\bf Case 3:} $\psi(x_*)=s/r$ with $s\in \{2,\ldots,r\}$. Since $\tilde{G}$ has 
a maximum at $(t_*,x_*)$ we have for both neighbors of $x_*$
\[ 
\psi(\eta_j(x_*)) Lv(\eta_j(x_*))\le \psi(x_*) Lv(x_*),
\]
that is
\[
Lv(t_*,\eta_j(x_*))\le \frac{\psi(x_*)}{\psi(\eta_j(x_*))} Lv(t_*,x_*)\le 
\frac{s/r}{(s-1)/r}Lv(t_*,x_*)=\eta Lv(t_*,x_*),
\]
with $\eta=s/(s-1)\in (1,2]$. By \autoref{lem-eta}, the second zero 
$a_*(\eta)>0$ of the function
$f(a)=e^{-\eta a}-1+a$ behaves as $2(\eta-1)$ as $\eta\to 1$. This implies that 
there exists $C_*>0$
independent of $r$ 
such that
\begin{align} \label{astern}
s a_*(\eta)= \frac{\eta}{\eta-1} a_*(\eta)\le C_*,\quad
\end{align}
for all $s\in \{2,\ldots,r\}$. We now distinguish two cases.

{\em Case 3a:} Suppose that $Lv(t_*,x_*)\le a_*(\eta)$. Then, by 
\eqref{astern},
\begin{align*}
\tilde{G}(t_1,x) & \le \tilde{G}(t_*,x_*) =\psi(x_*)\varphi(t_*)^{-1} 
Lv(t_*,x_*)\\
& \le \frac{s}{r}\varphi(t_*)^{-1}\,a_*(\eta) \le  
\frac{C_*}{r}\varphi(t_1)^{-1},\quad x\in \bar{B}_{2r}(x_0).
\end{align*}

{\em Case 3b:} Suppose that $Lv(t_*,x_*)>a_*(\eta) $. Then
\[
e^{-\eta Lv(t_*,x_*)}-1+Lv(t_*,x_*)>0.
\]
Now we use \eqref{infZ2}. 
At the maximum point, we can bound $\Theta(v)$
from below as follows.
\begin{align*}
\Theta(v)(t_*,x_*) 
& \ge \sum_{j=1}^2 e^{v(t_*,\eta_j(x_*))-v(t_*,x_*)}
\Big(e^{-\eta Lv(t_*,x_*)}-1+Lv(t_*,x_*)\Big)\\
& \ge  2e^{-\frac{Lv}{2}}
\Big(e^{-\eta Lv}-1+Lv\Big)\\
& = 2e^{-\frac{Lv}{2}}
\Big(e^{- Lv}-1+Lv\Big)+  2e^{-\frac{Lv}{2}}\Big(e^{-\eta Lv}-e^{-Lv}\Big)\\
& = 2e^{-\frac{Lv}{2}} \Upsilon(-Lv)+  2e^{-\frac{Lv}{2}}\Big(e^{-\eta 
Lv}-e^{-Lv}\Big).
\end{align*}

By convexity of the exponential function we have (at the maximum point)
\[
e^{-Lv}-e^{-\eta Lv}\le e^{-Lv}\big(-Lv+\eta Lv \big)=(\eta-1) Lv  e^{-Lv}.
\]
Now,
\[
\eta-1=\frac{s}{s-1}-1=\frac{1}{s-1},
\]
and thus we obtain at the maximum point (using \eqref{infZ2})
\[
F(Lv)=2e^{-\frac{Lv}{2}}\Big(\Upsilon(Lv)+\Upsilon(-Lv)\Big)\le  -\dot{\varphi} 
\varphi^{-1}Lv
+\frac{2}{s-1}e^{-\frac{Lv}{2}} Lv e^{-Lv}.
\]
Dividing by $Lv$ and using that $Lv>a_*(\eta)$ it follows that
\[
H(Lv)\le H(\varphi)+\frac{2}{s-1}e^{-\frac{3}{2}a_*(\eta)}.
\]
Setting 
\[
\omega_s=\frac{2}{s-1}e^{-\frac{3}{2}a_*(\eta)}
\]
and applying \autoref{lem-superadd} then gives
\[
H(Lv)\le H(\varphi)+\omega_s=H(\varphi)+H(H^{-1}(\omega_s))\le 
H\big(\varphi+\gamma H^{-1}(\omega_s)\big), 
\]
and thus
\[
Lv(t_*,x_*)\le \varphi(t_*)+\gamma H^{-1}(\omega_s),
\]
that is
\[
G(t_*,x_*)\le 1+\varphi(t_*)^{-1}\gamma H^{-1}(\omega_s).
\]
We now find that
\begin{align*}
\tilde{G}(t_1,x) & \le \tilde{G}(t_*,x_*)=\frac{s}{r}G(t_*,x_*)\\
& \le \frac{s}{r}\big(1+\varphi(t_*)^{-1}\gamma H^{-1}(\omega_s)\big)\\
& \le 1+\frac{\gamma C_s}{r} \varphi(t_1)^{-1},
\end{align*}
where $C_s:=s H^{-1}(\omega_s)$. It is now not difficult to see that there 
exists a number $M>0$ such that $C_s\le M$ for all $s\ge 2$ (recall that $H(x)$ 
behaves as a linear function as $x\to 0$).
It follows that
\[
\tilde{G}(t_1,x)\le  1+\frac{\gamma M}{r} \varphi(t_1)^{-1}.
\]

Collecting the estimates from all cases we see that for arbitrary $t_1>0$
\[
\tilde{G}(t_1,x)\le  1+\max\big\{1,\gamma H^{-1}(2/\sqrt{e^3}),C_*,
\gamma M \big\}
\,\frac{ \varphi(t_1)^{-1}}{r},\quad x\in 
\bar{B}_{2r}(x_0),
\]
which implies that
\[
Lv(t,x)\le  \varphi(t)+\max\big\{1,\gamma H^{-1}(2/\sqrt{e^3}),C_*,
\gamma M \big\}\,\frac{1}{r},\quad t>0,\,x\in 
\bar{B}_{r}(x_0).
\]
This together with \eqref{infZ0} proves the theorem.
\end{proof}

From \autoref{thm-Z} we obtain the following result for global positive solutions of the heat
equation on the grid $\tau \iZ$.

\begin{corollary}\label{cor:tau-business}
Let $\tau>0$ and $G$ be the grid $\tau\iZ$. Consider the Laplace operator given by
\[
\Delta_\tau u(x)=\frac{1}{\tau^2}\big(u(x+\tau)-2u(x)+u(x-\tau)\big),\quad x\in \tau\iZ.
\]
Suppose that $u:[0,\infty)\times \tau\iZ\to (0,\infty)$ solves the heat equation on $\tau\iZ$.
Then
\begin{align}\label{eq:tau-business}
-\Delta_\tau (\log u)(t,x)\le \varphi_\tau(t) \quad \mbox{on}\;(0,\infty)\times 
\tau\iZ,
\end{align}
where 
\[
\varphi_\tau(t)=\frac{1}{\tau}\,\varphi\big(\frac{t}{\tau}\big)
\]
is the relaxation function corresponding to the CD-function $F$
given in \eqref{RicciF} with $D=2$ and $\mu_0=\tau$, and $\varphi$ is as in  \autoref{thm-Z}.
\end{corollary}

\begin{remark}
If one considers the limit $\tau \to 0+$ in \eqref{eq:tau-business}, one 
recovers the classical sharp Li-Yau inequality 
\begin{align}\label{eq:tau-limit}
- \frac{\partial ^2}{\partial x^2} (\log u)(t,x)\le \frac{1}{2t}\quad 
\mbox{on}\;(0,\infty)\times \R  \,.
\end{align}
This follows from the fact that $\varphi(s) \sim \frac{1}{2s}$ as $s \to 
\infty$, cf. \autoref{lem:asymptrelax}.

\end{remark}

\begin{proof}[Proof of \autoref{cor:tau-business}]
The case $\tau=1$ follows directly from \autoref{thm-Z} by sending $R\to \infty$. The case of arbitrary
$\tau>0$ is reduced to the case $\tau=1$ by means of a scaling argument. In fact, putting
$t=s \tau^2$, $x=\tau y$ and $w(s,y)=u(t,x)$, the function $w$ solves the equation with Laplace operator
$\Delta=\Delta_1$ on $\iZ$, and thus
\[
-\Delta (\log w)(s,y)\le \varphi(s) \quad \mbox{on}\;(0,\infty)\times \iZ.
\]
Scaling back to the original variables yields the result.
\end{proof}

\section{Harnack inequalities}\label{sec:harnack}

The aim of this section is to provide a proof of the Harnack inequality. The 
case of finite graphs, \autoref{theo:Harnack_ineq_finite}, follows from the 
more general case, which we formulate here.

\begin{theorem} \label{theo:ersteHarnack}
Let $G=(V,E)$ be a connected and locally finite graph and 
$\mu:\,V\rightarrow (0,\infty)$ be bounded above by $\mu_{max}$. Let further 
$w_{min}>0$ be a lower bound for all weights $w_{xy}$ with $xy\in E$.

Suppose that
$u:\,(0,\infty)\times V\rightarrow (0,\infty)$ is $C^1$ in time and satisfies 
the differential
Harnack estimate
\begin{align} \label{DHassumption}
\partial_t (\log u)\ge \Psi_\Upsilon(\log u)-\eta(t)\quad 
\mbox{on}\;(0,\infty)\times V,
\end{align}
where $\eta:\,(0,\infty)\rightarrow [0,\infty)$ is continuous. Then for any 
$0<t_1<t_2$ and $x_1, x_2\in V$
we have
\begin{align} \label{HarnackI}
u(t_1,x_1)\le u(t_2,x_2) \exp\Big(\int_{t_1}^{t_2} \eta(t)\,dt+\frac{2\mu_{max} 
d(x_1,x_2)^2}{w_{min} (t_2-t_1)} \Big).
\end{align}
\end{theorem}

In the proof, we closely follow the strategy of \cite[Theorem 5.2]{BHLLMY15}.

\begin{proof} We first consider the situation where $x_1\sim x_2$. Let 
$0<t_1<t_2$ and $s\in J:= [t_1,t_2]$.
Then we have by assumption (\ref{DHassumption}) that
\begin{align}
\log \frac{u(t_1,x_1)}{u(t_2,x_2)} & =\log \frac{u(t_1,x_1)}{u(s,x_1)}+\log 
\frac{u(s,x_1)}{u(s,x_2)}
+\log \frac{u(s,x_2)}{u(t_2,x_2)}\nonumber\\
& = -\int_{t_1}^s \partial_t \log u(t,x_1)\,dt+\log 
\frac{u(s,x_1)}{u(s,x_2)}-\int_{s}^{t_2} \partial_t \log u(t,x_2)\,dt\nonumber\\
& \le \int_{t_1}^s \big(\eta(t)-\Psi_\Upsilon(\log u)(t,x_1)\big)\,dt+\log 
\frac{u(s,x_1)}{u(s,x_2)}\nonumber\\
& \quad +\int_{s}^{t_2} \big(\eta(t)-\Psi_\Upsilon(\log 
u)(t,x_2)\big)\,dt\nonumber\\
& \le \int_{t_1}^{t_2} \eta(t)\,dt+\log 
\frac{u(s,x_1)}{u(s,x_2)}-\int_{s}^{t_2}\Psi_\Upsilon(\log 
u)(t,x_2)\,dt\nonumber\\
& \le \int_{t_1}^{t_2} \eta(t)\,dt+\log 
\frac{u(s,x_1)}{u(s,x_2)}-\frac{w_{min}}{\mu_{max}}\,\int_{s}^{t_2}
\Upsilon\big(\log u(t,x_1)-\log u(t,x_2)\big)\,dt\nonumber\\
& = \int_{t_1}^{t_2} \eta(t)\,dt+\delta(s)-\gamma\int_{s}^{t_2}
\Upsilon\big(\delta(t)\big)\,dt,
 \label{HarStep1}
\end{align}
where we set $\gamma=w_{min}/\mu_{max}$ and
\[
\delta(t)=\log u(t,x_1)-\log u(t,x_2),\quad t\in J.
\]
We choose $s\in J$ in such a way that the continuous function $\omega$ defined 
by
\[
\omega(t):=\delta(t)-\gamma\int_{s}^{t_2} \Upsilon\big(\delta(t)\big)\,dt,\quad 
t\in J,
\]
attains its minimum at $s$. 

Suppose that $\omega(s)\ge 0$. Then the positivity of $\Upsilon$ implies that 
$\delta\ge 0$ in 
$J$, and thus
\[
\Upsilon\big(\delta(t)\big)\ge \frac{1}{2}\,\delta(t)^2,\quad t\in J,
\]
since $\Upsilon(z)\ge z^2/2$ for all $z\ge 0$. Putting
\[
\tilde{\omega}(t):=\delta(t)-\frac{\gamma}{2}\,\int_{s}^{t_2} 
\delta(t)^2\,dt,\quad t\in J,
\]
it follows that $\omega(s)\le \min_{t\in J} \tilde{\omega}(t)$. From Lemma 5.5 
in \cite{BHLLMY15}
we now know that
\[
\min_{t\in J} \tilde{\omega}(t)\le \,\frac{2}{\gamma(t_2-t_1)}.
\]
Combining the last two inequalities and \eqref{HarStep1} yields
\begin{align} \label{HarStep2}
\log \frac{u(t_1,x_1)}{u(t_2,x_2)}\le \int_{t_1}^{t_2} 
\eta(t)\,dt+\,\frac{2}{\gamma(t_2-t_1)}.
\end{align}

Now we consider the case when $x_1$ and $x_2$ are not adjacent. Set 
$l=d(x_1,x_2)$. Since $G$
is connected, there is a path $x_1=y_0\sim y_1\sim \ldots \sim y_l=x_2$ of 
length $l$. Define the
numbers $\tau_i$, $i=0,\ldots,l$ by $\tau_i=t_1+i(t_2-t_1)/l$. Employing 
\eqref{HarStep2} we may estimate as follows.
\begin{align*}
\log \frac{u(t_1,x_1)}{u(t_2,x_2)} & = \sum_{i=1}^l \log 
\frac{u(\tau_{i-1},y_{i-1})}{u(\tau_i,y_i)} \\
& \le  \sum_{i=1}^l \big(\int_{\tau_{i-1}}^{\tau_i} 
\eta(t)\,dt+\,\frac{2}{\gamma(\tau_i-\tau_{i-1})}\big)\\
& =  \int_{t_1}^{t_2} \eta(t)\,dt+\frac{2 l^2}{\gamma (t_2-t_1)}.
\end{align*}
This implies \eqref{HarnackI}. 
\end{proof}

\medskip

Recall the definition $\Upsilon_\alpha(y)=\Upsilon(\alpha y)$.

\begin{theorem} \label{zweiteHarnack}
Let $G=(V,E)$, $\mu_{max}$ and $w_{min}$ as in Theorem \ref{theo:ersteHarnack}, 
and 
let $\alpha\in (0,1)$.

Suppose that
$u:\,(0,\infty)\times V\rightarrow (0,\infty)$ is $C^1$ in time and satisfies 
the differential
Harnack inequality
\begin{align*} 
\partial_t (\log u)\ge \Psi_\Upsilon(\log 
u)-\frac{1}{\alpha}\Psi_{\Upsilon_\alpha}(\log u)-\eta_\alpha(t)\quad 
\mbox{on}\;(0,\infty)\times V,
\end{align*}
where $\eta_\alpha:\,(0,\infty)\rightarrow [0,\infty)$ is continuous. Then for 
any $0<t_1<t_2$ and $x_1, x_2\in V$
we have
\begin{align*} 
u(t_1,x_1)\le u(t_2,x_2) \exp\Big(\int_{t_1}^{t_2} 
\eta_\alpha(t)\,dt+\frac{2\mu_{max} d(x_1,x_2)^2}{w_{min} (1-\alpha)(t_2-t_1)} 
\Big).
\end{align*}
\end{theorem}

\begin{proof} The arguments are the same as in the proof of Theorem 
\ref{theo:ersteHarnack}.
Instead of \eqref{HarStep1}, we obtain in the first step the estimate
\[
\log \frac{u(t_1,x_1)}{u(t_2,x_2)}\le \int_{t_1}^{t_2} 
\eta_\alpha(t)\,dt+\delta(s)-\gamma\int_{s}^{t_2}
\Big(\Upsilon\big(\delta(t)\big)-\frac{1}{\alpha}\Psi_{\Upsilon_\alpha}
\big(\delta(t)\big)\Big)\,dt.
\]
From Lemma \ref{alphalemma}, we know that the function 
$g(z):=\Upsilon(z)-\frac{1}{\alpha}\,\Upsilon(\alpha z)$ is nonnegative on $\iR$ 
and satisfies
$g(z)\ge (1-\alpha)z^2/2$ on $[0,\infty)$. Therefore, we can argue as above, 
replacing $\gamma$
by $\gamma(1-\alpha)$.
\end{proof}


\newcommand{\etalchar}[1]{$^{#1}$}

\end{document}